\documentclass[a4paper,12pt]{article}

\usepackage[T1]{fontenc}  
\usepackage[ansinew]{inputenc}  
\usepackage{amsthm,amsmath,stmaryrd,bbm,hyperref,geometry,color,xcolor}
\usepackage{amssymb}
\usepackage[english]{babel}
\usepackage{graphicx}
\usepackage{amsfonts,amssymb}
\usepackage{verbatim}
\usepackage{enumitem}
\usepackage[all]{xy}
\usepackage{caption}
\usepackage{subcaption}

\newcommand{\lj}[1]{{\color{black}#1}}

\newcommand{\po}{\left(}
\newcommand{\pf}{\right)}

\newcommand{\E}{\mathbb E}
\newcommand{\R}{\mathbb R} 
\newcommand{\T}{\mathbb T} 
\newcommand{\C}{\mathcal C}
\newcommand{\D}{\mathcal D}
\newcommand{\X}{\mathcal X}

\newcommand{\N}{\mathbb N} 
\newcommand{\B}{\mathcal B}

\newcommand{\dd}{\text{d}}
\newcommand{\na}{\nabla}
 
\newcommand{\Hd}{\mathcal H_{\delta,\ell}}

\newtheorem{thm}{Theorem}
\newtheorem{assu}{Assumption}
\newtheorem{lem}[thm]{Lemma}

\newtheorem{prop}[thm]{Proposition}
\newtheorem{cor}[thm]{Corollary}

\title{Weak error expansion of a numerical scheme with rejection for singular Langevin process}
\author{Lucas Journel}
\date{ }

\begin{document}

\maketitle

\begin{abstract}
    We show expansion \textit{\`a la Talay-Tubaro} of a numerical scheme with rejection for the Langevin process in the case of a singular potential. In order to achieve this, we provide estimates on the associated semi-group of the process. The class of admissible potentials includes the Lennard-Jones interaction with confinement, which is an important potential in molecular dynamics and served as the primary motivation for this study.
\end{abstract}

\section{Introduction}\label{sec:intro}

Many physical systems, can be described by a Hamiltonian (or energy) of the form:
\begin{equation}
    H(x,y) = U(x) + \frac{1}{2}|y|^2,
\end{equation}
where $U:\R^d\to \R_+$ \lj{(or $U:\T^d\to \R_+$, where $\T^d$ is the $d$-dimensional torus)} is a potential function, and $|y|^2/2$ represents the kinetic energy. When the dimension $d$ is large, obtaining the exact evolution of the system becomes infeasible, and a statistical description must be employed. In statistical physics, for systems with fixed temperature $T=\beta^{-1}>0$ and size, but that can exchange energy, the system is characterized by a probability measure known as the Gibbs measure. The probability for the system to be in a state $(x,y)\in\R^{2d}$ solely depends on its energy and is given by 
\begin{equation}\label{def:gibbs}
    \mu(\dd x,\dd y) = \frac{e^{-\beta H(x,y)}}{Z} \dd x\dd y,
\end{equation}
where $Z$ is a normalization constant. The goal of molecular dynamic is to compute macroscopic quantities of the system defined by averages of the form  
\begin{equation}\label{eq:average}
    \mu(f) := \mathbb E(f(X,Y)),\quad (X,Y)\sim \mu,
\end{equation}
for some function $f:\R^d\to\R$. One important example in chemistry is a system of $N\gg 1$ particles with van der Waals interaction, which corresponds to a noble gas. In this case, $d=3N$, and $x=(x_1,\dots,x_N)$ represents the positions of the particles. The potential can be written as
\[
U(x) = \sum_{i=1}^N U_c(x_i) +  \sum_{1\leqslant i\neq j\leqslant N} U_i(|x_i-x_j|),
\]
where $U_c:\R^{3N}\to\R_+$ is a confining potential and for $r\in\R^3$,
\[
U_i(r) = 4\po \frac{1}{|r|^{12}} - \frac{1}{|r|^{6}} \pf.
\]
The interaction part $U_i$ is commonly known as the Lennard-Jones potential. In the case where the gas is confined in a volume $D$, one could then get its pressure by computing the average
\[
P=\mu_{\beta}(f_P),\qquad f_P(x,y) = \frac{1}{3|D|}\sum_{i=1}^N \left( \frac{y_i^2}{m} - x_i\cdot \na_{x_i} U(x) \right),
\]
where $m> 0$ is the mass of a particle.
What makes this example interesting is the presence of singularities in the potential, as it can take infinite value, and the fact that the function $f_P$ is not locally integrable. \lj{For a more in-depth introduction to molecular dynamics, refer to~\cite{tuckerman2023statistical}, as well as~\cite{dynamol-tony-gab} for a mathematical perspective.}

Many algorithms have been developed to compute the average~\eqref{eq:average}, and a review of these algorithms can be found in~\cite{simu_gibbs}. Among them is the kinetic Monte-Carlo method, which is based on the so-called kinetic Langevin process given by:
\begin{equation}\label{eq:langevin}
    \left\{ \begin{aligned} &\dd X_t = Y_t\dd t, \\ &\dd Y_t = -\nabla U(X_t)\dd t  - \gamma Y_t \dd t + \sqrt{2\gamma\beta^{-1}}\dd B_t, \end{aligned}\right.
\end{equation}
where $\gamma>0$ is the friction. This process models the motion of a particle in a potential $U$, in contact with a heat bath at a constant temperature $T=\beta^{-1}$. Under mild assumption, it admits $\mu$ as its unique stationary measure. In fact, it was shown in~\cite{ergodicity2} that this process is ergodic for a certain class of singular potential, meaning that:
\begin{equation}\label{eq:ergo}
    \E\po f \po X_t,Y_t\pf\pf \underset{t\rightarrow\infty}{\rightarrow}\mu(f),
\end{equation}
for $f$ in a suitable class of functions. Consequently, if we could simulate a large number of independent copies $(X^k_t,Y^k_t)_{1\leqslant k\leqslant n}$ of $(X_t,Y_t)$, then
\begin{equation}\label{eq:approx}
\frac{1}{n}\sum_{k=1}^n f(X^k_t,Y^k_t) \approx \mu(f),
\end{equation}
for a sufficiently large $n$ and $t$. Knowing the speed of convergence in~\eqref{eq:ergo} and in the law of large numbers, we are left with the study of the weak convergence of numerical schemes for this Langevin process to get the error made by~\eqref{eq:approx}. For non-singular potentials, there are numerous numerical schemes available, see for example~\cite{splitting_scheme} for a review on splitting schemes. The goal here is to adapt such a scheme and prove a series expansion in the weak convergence for singular potential, focusing on the following one:
\begin{equation}\label{def:schema_stoped}
    \left\{ \begin{aligned} &\bar X_{n+1} = \bar X_n + \mathbbm{1}_{E_{\delta}(\bar X_n,\bar Y_n,G_n)\in\Hd}\delta \bar Y_{n+1}, \\ &\bar Y_{n+1} = \bar Y_n +\mathbbm{1}_{E_{\delta}(\bar X_n,\bar Y_n,G_n)\in\Hd}\left(-\delta\nabla U(\bar X_n)  - \delta\gamma \bar Y_{n}  + \sqrt{2\gamma\beta^{-1}\delta}G_n\right) \end{aligned} \right.
\end{equation}
for some time step $\delta>0$, where $(G_n)$ is a family of independent standard Gaussian random variable, and where for a small fixed parameter $\lj{\ell}>0$,\lj{
\begin{equation}\label{eq:domaine_scheme}
\Hd = \left\{ \phi\leqslant \delta^{-\ell} \right\},
\end{equation}
with
\begin{equation}\label{eq:phi}
\phi(x,y) = H(x,y)+ 4d\gamma\beta^{-1} \frac{y\cdot\na U(x)}{1+|\na U(x)|^2},
\end{equation}
and}
\[
E_{\delta}(x,y,g) = \po x + \delta \po y - \delta\na U(x) -\delta\gamma y + \sqrt{2\gamma\beta^{-1}\delta}g \pf, y - \delta\na U(x) -\delta\gamma y + \sqrt{2\gamma\beta^{-1}\delta}g \pf.
\]
\lj{The definition of $\phi$ stems from the fact that, under our assumptions, for all $b<\beta$, $e^{b\phi}$ is a Lyapunov function for the Langevin process~\eqref{eq:langevin} (see Proposition~\ref{prop:Lyapunov} below), as well as for this numerical scheme (as proved in Lemma~\ref{lem:Lyapunov}).}
Hence, this numerical scheme is a version with rejection of the following (first-order) splitting scheme
\begin{equation}\label{def:schema_Euler}
    \left\{ \begin{aligned} &\tilde X_{n+1} = \tilde X_n + \delta \tilde Y_{n+1}, \\ &\tilde Y_{n+1} = \tilde Y_n -\delta\nabla U(\tilde X_n)  - \delta\gamma \tilde Y_{n}  + \sqrt{2\gamma\beta^{-1}\delta}G_n, \end{aligned}\right.
\end{equation}
but we only accept a step of the scheme if its \lj{Lyapunov function} doesn't exceed some threshold depending on the time step. We could have considered a simple Euler-Maruyama scheme, but if the initial condition satisfies $U(X_0 + \gamma Y_0)=\infty$, then the chain would not be well defined. The splitting scheme~\eqref{def:schema_Euler} is the simplest numerical scheme that is well-defined with probability $1$ for all initial conditions, and under a one-sided Lipschitz condition on $U$, it can be shown that this Markov chain is ergodic. However, there is another problem, which is that even if the numerical scheme is almost-surely well defined, $\tilde X_n$ has a density bounded below by the Lebesgue measure on all compact sets. Hence, if $f$ is not locally integrable, then :
\[
\E(f(\tilde X_n)) = \infty,\quad \forall n\geqslant 1.
\]
This is also the case for all splitting schemes, and it poses two problems. Firstly, it restricts us to consider averages only against bounded functions, which excludes the computation of important quantities such as the mean energy for the Lennard-Jones potential or the pressure of noble gases. Secondly, it prevents the construction of a Lyapunov function, which is a critical step in the proof of weak convergence for numerical schemes. To address these issues, we introduced the numerical Scheme~\eqref{def:schema_stoped}, which has better integrability properties, see Section~\ref{sec:scheme}.
Any scheme used for bounded smooth potential (e.g. splitting schemes) can be adapted in a similar manner (namely with rejection) for singular potentials. We study this specific first-order splitting scheme for simplicity. In the context of singular potentials in high-dimensional spaces, an implicit scheme may be ill-defined, and in any case would have a prohibitive numerical cost. Hence, we consider only explicit schemes. \lj{Another possible numerical scheme would have been to use rejection at a threshold depending only on the energy, i.e. replacing $\left\{ \phi\leqslant \delta^{-\ell} \right\}$ by $\left\{ H\leqslant \delta^{-\ell} \right\}$. However, since the Hamiltonian has not the same level line as any Lyapunov function of the Langevin process (such as $e^{b\phi}$, $b<\beta$), it is not clear that this numerical scheme admits a Lyapunov function, and hence satisfies the integrability property needed for the proofs. Notice in particular that $\na U$ is already computed for the drift, hence no supplementary computations are needed to check if $\phi$ is low enough.}
While in practice, averages are often computed using schemes without rejection, there are cases where rejection could lead to better results. The discussion in Section~\ref{s-sec:numerics} likely elaborates further on this aspect.

There are two approaches for showing the convergence of a numerical scheme towards the continuous process. The first one is strong convergence, which focuses on the convergence of trajectories. However, since our goal is the computation of averages, we are interested in weak convergence. For $f:\R^d\to\R$ in a suitable class of functions and $t\geqslant 0$, we aim to show that:
\[
\lim_{\delta\rightarrow0}\E(f(\bar X_{n})) = \E(f(X_t)),
\]
where $n=\lfloor t/\delta \rfloor$. Our objective is to get a series expansion of this convergence, commonly known as expansion \textit{\`a la Talay-Tubaro}, from the seminal work~\cite{Talay-Tubaro}. Write $(P_t)$ for the semi-group of the Langevin process~\eqref{eq:langevin} defined for regular \lj{and integrable} enough function by:
\[
P_tf(x,y) = \E_{(x,y)}(f(X_t,Y_t)). 
\]
Methods to show weak convergence and series expansion of such convergence rely on estimates on $P_tf$ and its derivatives, typically of the form: for $f:\R^d\to\R$, which grows at most polynomially at infinity (as well as its derivative of all orders), for all multi-index $\alpha\in\N^{2d}$, there exists $C,k>0$ such that:
\[
|\partial^{\alpha}P_tf(x,y)|\leqslant C(1+|x|^k+|y|^k).
\]
Such estimates cannot hold in the singular setting. Therefore, the primary objective of this work is to establish equivalent estimates in the singular setting, which will be addressed in Section~\ref{sec:hypocoercivity}. Those estimates will then enable us to prove in Section~\ref{sec:scheme} an expansion \textit{\`a la Talay-Tubaro} for the numerical scheme~\eqref{def:schema_stoped}, as well as for its invariant measure. In Section~\ref{sec:settings}, we will present the assumptions, theorems, and comparisons with existing works.

\section{Mathematical setting and results}\label{sec:settings}

\subsection{Assumptions and mains results}\label{s-sec:main_results}

Fix a potential $U:\R^d\to[0,\infty]$ and write
\begin{equation}\label{eq:domain}
    \D = \left\{x\in\R^d,\quad U(x)<\infty \right\}, \quad \X = \D\times \R^d, 
\end{equation}
for the domain of definition of the process $(X,Y)$ solution to equation~\eqref{eq:langevin}. In all of this work, we will write $z=(x,y)\in\X$ for the global variable, and $|\cdot|$ for the $l^2$-norm on $\R^d$. For any Banach space $\mathcal B$, let $\mathcal C^{\infty}(\mathcal B,\R)$ denote the set of smooth functions from $\mathcal B$ to $\R$, $\mathcal C_b^{\infty}(\mathcal B,\R)$ the set of bounded smooth functions, and $\mathcal C_c^{\infty}(\mathcal B,\R)$ the set of smooth functions with compact support. Let $\mathcal M^1(\R^{2d})$ denote the set of probability measures on $\R^{2d}$.
We work under several sets of assumptions. First:
\begin{assu}\label{assu:basic}
\begin{itemize}
    \item $U\in\C^{\infty}(\D,\R_+)$.
    \item $\D$ is connected and the set:
    \[\D_n = \left\{x\in\R^d,\quad U(x)<n \right\}\]is bounded for all $n\in\N$.
    \item $U$ defines a Gibbs measure at temperature $\beta$:\[\int_{\D}e^{-\beta U(x)}\dd x<\infty.\]
\end{itemize}
\end{assu}

\begin{assu}\label{assu:growth}
\begin{itemize}
    \item The following limit holds:
    \[
    \lim_{x\rightarrow \mathcal D^c} \frac{|\na^2 U(x)|}{|\na U(x)|^2}=0.
    \]
    \item There exists $c_0,c_{\infty},d_0,d_{\infty}>0$, $\eta_0\in \R\setminus [-1,0]$, $\eta_{\infty}>1$ such that:
    \[
    c_{\infty}U^{2-\frac{2}{\eta_{\infty}}}+d_{\infty}\leqslant |\na U|^2\leqslant c_0U^{2+\frac{2}{\eta_0}}+d_0.
    \]
\end{itemize}
\end{assu}

\begin{assu}\label{assu:Hr-setting}
For all $\alpha\in\N^{d}$, there exist $C_{\alpha}>0$, $k_{\alpha}\in\N$ such that
\[ 
|\partial^{\alpha}U| \leqslant U^{k_{\alpha}} + C_{\alpha}.  
\]
\end{assu}

Let $\mathcal P(U)$ denote the set of admissible functions for our theorems:
\[
\mathcal P(U) = \left\{ f\in\mathcal C^{\infty}(\mathcal X,\R)|\forall \alpha\in\N^{2d},\,\exists C,c>0,\, c<\beta,\, |\partial^{\alpha} f|\leqslant Ce^{cH}\right\}.
\]

The first theorem gives the estimates on the semi-group necessary for the proof of weak convergence.

\begin{thm}\label{thm:estimates}
Suppose Assumptions~\ref{assu:basic},~\ref{assu:growth} and~\ref{assu:Hr-setting}.
Then for all $f\in\mathcal P(U)$, and all multi-index $\alpha\in\N^{2d}$, there exist $C,q>0$, $b\in (\beta(1-\frac{1}{2d}),\beta)$, such that for all $t>0$, $z\in\mathcal X$
\[
    |\partial^{\alpha}(P_tf-\mu(f))(z)| \leqslant Ce^{-qt}e^{bH(z)}.
\]
\end{thm}

This theorem allows for the proof of the two following theorems:

\begin{thm}\label{thm:num_approx_finite_time}
Suppose Assumption~\ref{assu:basic},~\ref{assu:growth} and~\ref{assu:Hr-setting}. Then there exists $l_0>0$ such that for all $0<\ell<\ell_0$, $f\in\mathcal P(U)$, and $t\geqslant 0$, there exists a family $(C_i(t))_i$ of explicit real numbers such that:
\[
\E_{z}(f(\bar Z_{n})) = \E_{z}(f(Z_t)) + C_1(t)\delta + \cdots + C_{k}(t)\delta^{k} + O(\delta^{k+1}),
\]
where $(\bar Z_n)_n=(\bar X_n,\bar Y_n)_n$ is the numerical scheme defined in~\eqref{def:schema_stoped}, $n\delta=t$, and $O(\delta^{k+1})$ is uniform in $t$.
\end{thm}

\begin{thm}\label{thm:num_approx_stationary}
Suppose Assumption~\ref{assu:basic},~\ref{assu:growth} and~\ref{assu:Hr-setting}. Then there exists $\delta_0,\ell_0>0$ such that for all $0<\delta<\delta_0$, $0<\ell<\ell_0$, the numerical scheme defined in~\eqref{def:schema_stoped} admits an unique invariant measure $\mu_{\delta}\in\mathcal M^1(\R^{2d})$. For all $f\in \mathcal P(U)$, there exists $C,c>0$ such that for all $z\in\Hd$:
\[
\left|\E_z\po f \po \bar Z_n\pf \pf - \mu_{\delta}(f)\right|\leqslant Ce^{-cn}.
\]
Moreover, there exists a family $(\tilde C_i)_i$ of explicit real numbers such that:
\[
\mu_{\delta}(f) = \mu(f) + \tilde C_1\delta + \cdots + \tilde C_{k}\delta^{k} + O(\delta^{k+1}).
\]
\end{thm}

Let's comment on those assumptions and theorems. Assumptions~\ref{assu:basic} and~\ref{assu:growth} are derived from~\cite{BaudoinGordinaHerzog}, on which this work is based. As explained in their work, Assumption~\ref{assu:basic} is the minimal requirement to ensure the pathwise well-posedness of the process~\eqref{eq:langevin}, as well as for the Gibbs measure to be well-defined and a stationary measure of the Langevin process. Assumption~\ref{assu:growth} differs slightly from the one in~\cite{BaudoinGordinaHerzog}, as they suppose that there exists $\kappa>0$ such that for all $v\in\R^d$
\begin{equation}\label{eq:assu_not_used}
|\na^2U(x)v|\leqslant \frac{\beta}{16d} |\na U(x)|^2|v| + \kappa |v|,
\end{equation}
which would allow potentials that exhibit logarithmic singularities. Here, we forbid those potentials, following the assumption stated in~\cite{ergodicity2}. This assumption plays a vital role in the construction of a Lyapunov function for both the continuous-time process and the numerical scheme.  We need a Lyapunov function of order $e^{bH}$ for any $b<\beta$, which cannot be achieved under~\eqref{eq:assu_not_used}. In this case, it would be possible that the bounds on the semi-group given by Theorem~\ref{thm:estimates} are not in $L^1(\mathcal Law(\bar Z_n))$. Since the proof of theorem~\ref{thm:estimates} relies on computation in Sobolev spaces, and the use of Sobolev embedding, as in~\cite{Hk-hypo,journelannealing}, we need an additional assumption on the derivatives of the potential of all orders in order to carry out the computation. However, the Villani-type condition presented in~\cite{Hk-hypo} is not satisfied by singular potentials, as in the $H^1$ case from~\cite{BaudoinGordinaHerzog}. Therefore, we impose Assumption~\ref{assu:Hr-setting}. 
\lj{This set of assumptions encompasses any repulsive interaction that exhibits a sufficiently rapid explosion (at least algebraic), such as Lennard-Jones interaction, and Coulomb interaction as soon as the particles are living in $\R^q$ for $q\geqslant 3$, coupled with an additional confinement potential. However, it's important to note that 2-dimensional Coulomb interactions ($U_i(r) \propto \ln(r)$) do not satisfy these assumptions.}

It is well-established in numerical probability that the proof of weak convergence relies on estimates of the kind given by Theorem~\ref{thm:estimates}. The proof of this theorem is based on Gamma calculus, see Section~\ref{s-sec:Gamma_Calculus}. This enables us to make series expansion at any order of the error induced by the numerical scheme at each step, with a remaining term of the form $\delta^k\E\po e^{bH(\bar Z_n)} \pf$. This would be infinite in the case of a numerical scheme without any rejection mechanism like~\eqref{def:schema_Euler}, but we will show that in the case of the numerical scheme~\eqref{def:schema_stoped}, there is, under our assumptions, as in the continuous setting, a Lyapunov function of this order, uniformly over small time step. In particular, we may apply the method developed in~\cite{Talay-Tubaro} to get weak convergence as well as expansion \textit{\`a la Talay-Tubaro} for this numerical scheme. \lj{It can be noticed from the proofs that the constants $(C_k(t))_k$ and $(\tilde C_k(t))_k$ are independent of $\ell>0$, and can be expressed in the same way as if we considered the numerical scheme~\eqref{def:schema_Euler} with a regular potential, as in~\cite{splitting_scheme}. For instance, we have the following formula:
\[
C_1(t) = \int_0^t \E_z\left( \psi(s,Z_{t-s}) \right) \dd s,
\]
where
\begin{multline*}
    \psi(t,z) = \po \na U + \gamma y \pf \cdot \po \na_y/2 - \na_x\pf  P_tf + \frac{1}{2}\partial_t^2 P_tf - \frac{1}{2}y\cdot \na_x\partial_t P_tf + \frac{1}{2}\po \na U + \gamma y \pf \cdot \na_y^2y \po \na U + \gamma y \pf \\- \frac{1}{2}y\cdot \na_{x,y}^2 \po \na U + \gamma y \pf + \gamma\beta^{-1}\sum_{i=1}^d \partial_{x_i}\partial_{y_i}P_tf - \frac{2}{3}\gamma\beta^{-1} \partial_t\Delta_yP_tf \\- \frac{7}{12}\gamma\beta^{-1} \po \na U + \gamma y \pf\cdot \na_y\Delta_yP_tf + \frac{1}{3}\gamma\beta^{-1}y\cdot\na_x\Delta_y P_tf \\ + \frac{1}{2} \po\gamma\beta^{-1}\pf^2 \sum_{i=1}^d \partial^4_{y_i}P_tf + \frac{1}{6}\po \gamma\beta^{-1} \pf^2 \sum_{i\neq j=1}^d \partial^2_{y_i}\partial_{y_j}^2P_tf.
\end{multline*}}

Besides, a corollary of Theorem~\ref{thm:num_approx_finite_time} is the uniform in time weak convergence of the process. For a given $f\in\mathcal P(U)$, and $z\in\R^{2d}$, there exists $C>0$ such that for all $t\geqslant 0$, $n\delta = t$:
\[
\left|\E_{z}(f(\bar Z_{n})) - \E_{z}(f(Z_t))\right| \leqslant C\delta.
\]

Theorems~\ref{thm:num_approx_finite_time} and~\ref{thm:num_approx_stationary} are also motivated by the following fact:
\[
2\E\po f \po Z^{\delta/2}_{2n}\pf \pf - \E\po f \po Z^{\delta}_{n}\pf \pf = \E\po f \po Z_t\pf \pf - \frac{C_2}{2}\delta^2 + O(\delta^3),
\]
which yields a better order convergence. With arbitrary order expansion, it is possible to achieve any order of convergence using a combination of $\po \E\po f \po Z^{\delta/2^k}_{2^kn}\pf \pf \pf$, \lj{albeit at the expense of increased computational complexity. This approach is known as Romberg-Richardson interpolation. While this interpolation method could potentially reduce bias, it might also lead to an increase in the variance of the scheme.} For numerical experiments on this improvement, see the work of Talay and Tubaro~\cite{Talay-Tubaro}.

\subsection{Comparison between numerical scheme with and without rejection}\label{s-sec:numerics}

In practice, numerical studies often utilize splitting schemes with no rejection on the bounded torus. If the potential is defined on $\T^d = \po\R / \mathbb Z \pf^d$, then the Langevin process~\eqref{eq:langevin} would be defined on $\T^d\times\R^d$. All proofs would still be valid, preserving the theorems' validity. As discussed in the introduction, for the numerical scheme~\eqref{def:schema_Euler}, it is not possible to take the expectation of unbounded functions, and \lj{an invariant probability measure may not exist}. Additionally, it also remains an open problem as to whether finite time weak expansion would still hold under our assumptions for bounded functions and numerical schemes without rejection. Regarding the use of schemes without rejection in practical applications, it is justified by the fact that the continuous process does not reach the singularities. For any fixed time $T>0$, one can select a time step $\delta>0$ small enough so that the numerical scheme without rejection does not get close to the singularities, and the computation of averages would yield reasonable results. However, for a fixed $\delta$, the scheme would approach arbitrarily close to the singularity in large time, leading to abnormally large steps. Whereas with the rejection mechanism, it is possible to choose larger time steps, resulting in a more computationally efficient process, and avoiding the issue of getting too close to the singularities. This advantage of the rejection mechanism makes it a preferable choice in certain practical scenarios, such as in the following $1$-dimensional toy model:
\[
U(x) = \frac{1}{x} + x^2,
\]
for which we conducted the following numerical experiment: fix $\beta^{-1}=15$, $\gamma=1$, and a final time $T=15000$. Define the empirical averages by
\[
n \mapsto S_n = \frac{1}{n}\sum_{k=0}^{n-1} U(Y_k),
\]
where $Y_k = \bar X_k$ or $Y_k = \tilde X_k$. For different values of $\delta$, we simulate $K=1000$ copies $S^i$ of $S_{\lfloor T/\delta \rfloor}$. We plot the evolution of the proportion of copies that have more than $1\%$ error:
\[
\delta \mapsto \frac{1}{K}\sum_{i=1}^K \mathbbm 1_{S^i\notin [0.99\mu_\beta(U),1.01\mu_\beta(U)]},
\]
where we estimated $\mu_{\beta}(U)$ using \textit{Wolfram Alpha}. We get the result in Figure~\ref{fig:fixed_time}. 

\begin{figure}
\centering
\includegraphics[width=.7\linewidth]{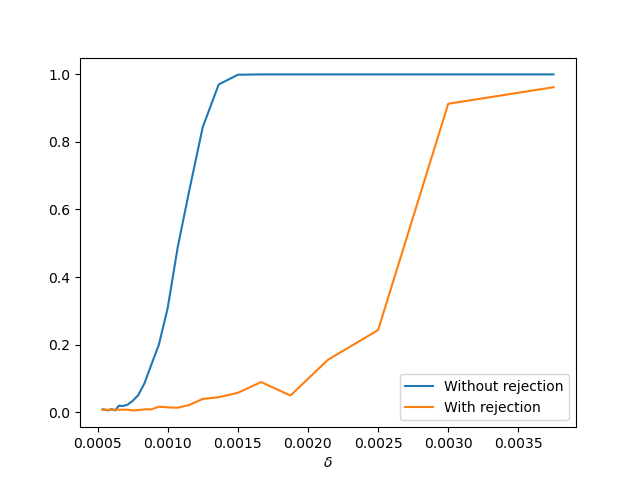}
\caption{Failure probability.}
\label{fig:fixed_time}
\end{figure}

We can observe that for very small $\delta$, the scheme without rejection produces satisfactory results. Nonetheless, there exists an interval of values of $\delta$ where the scheme with rejection significantly outperforms the other. The proportion of failures does not converge to $0$ because of the error stemming from insufficiently long simulation times. For values of $\delta$ that aren't small enough, the energy threshold cannot, to prevent explosion, be set high enough to ensure the convergence of the empirical averages converges to a value close to $\mu_{\beta}(U)$ ($\approx 9.035$). This behavior is illustrated in Figure~\ref{fig:one-traj}, where we display a typical trajectory for both cases, with $\delta=10^{-2}$.

\begin{figure}
\begin{subfigure}{.5\textwidth}
  \centering
  \includegraphics[width=.9\linewidth]{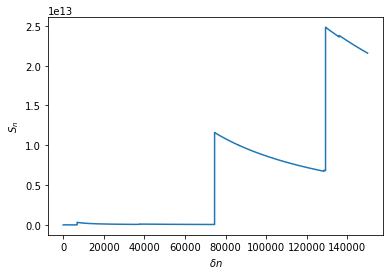}
  \caption{Without rejection.}
  \label{sfig:fixed_time_no_rej_16}
\end{subfigure}%
\begin{subfigure}{.5\textwidth}
  \centering
  \includegraphics[width=.9\linewidth]{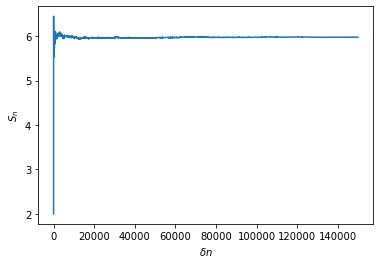}
  \caption{With rejection.}
  \label{sfig:fixed_time_no_rej_14}
\end{subfigure} 
\caption{$n \mapsto S_n$, $\delta=10^{-2}$.}
\label{fig:one-traj}
\end{figure}

For smaller $\beta$, the explosion phenomenon would be even more prominent because the variance of the Gaussian distribution is proportional to $\gamma\delta\beta^{-1}$. As $\beta$ decreases, the variance increases, making the numerical scheme more susceptible to divergent behavior and explosions in the trajectories. Hence, a smaller $\beta$ exacerbates the challenge of handling the singularities and maintaining stable and accurate numerical computations. This unstable behavior could also arise in the case of a metastable process. The metastability imposes long simulation times, which could prove longer than the time needed for the numerical scheme to visit the singularities.
The conclusion of this experiment/discussion is that the use of the rejection mechanism in practical applications may help to avoid divergent behavior and explosions, leading to more stable and accurate results.

\lj{An alternative approach to address the singularities is to employ a Metropolis-adjusted scheme, as the one presented in~\cite{M-H-Langevin}. Similarly to our scheme, Metropolis-adjusted algorithms tend to reject transitions going too close to singularities, but the difference is that they also reject moves in low-energy regions, where the process spends most of its time. For the toy model discussed in this section, such a scheme exhibits remarkably high performance, with a success rate exceeding 99\% for all considered time steps. However, in high-dimensional settings, Metropolis-adjusted schemes often result in numerous rejections and exhibit suboptimal performance. As a result, they are typically not used in the realms of molecular dynamics or statistics, even in the case of regular potentials.}

\subsection{Related works}\label{s-sec:related_work}

The weak error expansion of the numerical scheme, as stated in Theorem~\ref{thm:num_approx_finite_time} and Theorem~\ref{thm:num_approx_stationary}, was initially proven and numerically studied by Talay and Tubaro in~\cite{Talay-Tubaro}. Their work focuses on the Euler-Maruyama and Milstein schemes for Elliptic SDEs with globally Lipschitz coefficients. See also references therein for weak convergence of numerical schemes for such processes. Talay also proved similar expansions for Hamiltonian systems with polynomial growth at infinity in~\cite{Talay}, although using an implicit Euler scheme. Those expansions are also called weak backward analysis in~\cite{WBA1,WBA2,WBA3}. The long-term behavior of numerical schemes has been explored in~\cite{long_time_scheme}, where it is shown that the explicit Euler scheme might not be ergodic for non-globally Lipchitz vector field. See also~\cite{failure_euler} for the study of the failure of the Euler-Maruyama scheme. Additional references can be found in~\cite{scheme1,scheme2,scheme3}. Regarding explicit numerical schemes for the Langevin process with bounded potentials on the torus, see~\cite{splitting_scheme}, as well as~\cite{scheme_general_kinetic_energy} for the Langevin process with more general kinetic energies. In the same spirit as numerical schemes with rejection, the used of stopped scheme to transfer integrability properties of the continuous process to the numerical scheme already appeared in~\cite{exp_int_stopped_scheme}, see also~\cite{scheme_martin}. For uniform in time convergence of a numerical scheme, see~\cite{conv_unif_time_scheme}.
Working with modified Sobolev norm to study the long time behavior of hypocoercive process was first introduced in the seminal work~\cite{vil}. For the study of the singular Langevin process, \cite{ergodicity,ergodicity3} focus on the construction of solution and ergodicity, and \cite{ergodicity2,BaudoinGordinaHerzog,ergodicityL2} on the long time behavior of the process in Wasserstein, $H^1$ and $L^2$ distance, respectively. \cite{Lyapu_coulomb} addresses the case of Coulomb interaction, which corresponds to a potential with logarithmic singularities.
For alternative methods of sampling for the Gibbs measure, see~\cite{simu_gibbs}.

\section{Proof of Theorem~\ref{thm:estimates}}\label{sec:hypocoercivity}

The proof of Theorem~\ref{thm:estimates} relies on Sobolev embedding, and $H^k$-estimates. Similar results have been established for Langevin processes and close to quadratic potentials, in the sense that there exists $0<m\leqslant M$ such that $mI_d \leqslant \na^2 U \leqslant MI_d$, where $I_d$ stands for the identity matrix of size $d$, in~\cite{Talay}. Here we will use a method previously employed in~\cite{Hk-hypo,journelannealing}. For a given $f\in\mathcal P(U)$, our goal is to show bound of the form:
\[
\int_{\X}|\partial^{\alpha}\po P_tf e^{-bH} \pf|^2 \dd z \leqslant Ce^{-q t},
\]
for some $0<b<\beta$ and all $\alpha\in\N^{2d}$. Sobolev embedding would then yield Theorem~\ref{thm:estimates}. To do so, we define a norm $\|\cdot\|_k$ that dominates the $H^k$-Sobolev norm, but such that $t\mapsto \|P_tf - \mu(f)\|_k$ converges exponentially fast to $0$. However, contrary to the cited work, we need to take into account the singular potential by using a Lyapunov function. It was shown in~\cite{BaudoinGordinaHerzog} that the Langevin process~\eqref{eq:langevin} admits a family of Lyapunov functions $(V_{0,\bar b})$ such that for all $0 < \bar b < \beta/(2d)$, $\varepsilon>0$:
\[C^{-1}e^{(1-\varepsilon)\bar b H}<V_{0,\bar b}<Ce^{(1+\varepsilon)\bar b H},\]
see~\cite[Equation 5.2]{BaudoinGordinaHerzog} for its exact definition and Proposition~\ref{prop:Lyapunov} below for its main properties. Thanks to this family of Lyapunov functions, we may now define our modified Sobolev norms. Let $\na^l_x\na_y^ph$ denote the vector of all derivative of $h$ of order $l$ on $x$ and $p$ on $y$:
\[
\na_x^{l}\na_y^{p}h=\left\{\partial^{\alpha_1}_x\partial^{\alpha_2}_yh\;\middle|\; |\alpha_1|=l,\, |\alpha_2|=p\right\},
\]
and its norm
\[
|\na_x^l\na_y^{p}h|^2=\sum_{|\alpha_1|= l;|\alpha_2|= p} |\partial_x^{\alpha_1}\partial_y^{\alpha_2} h|^2.
\]
Fix $0 < \bar b < \beta/(2d)$, and $k\in\N^*$. Write $V=V_{0,\bar b}^{1/k}$, and for $p\in \llbracket 1, k \rrbracket$, $W_p = V^p + \lambda$, for some $\lambda >0$. For $h\in \mathcal C^{\infty}(\X,\R)$, denote:
\begin{align}\label{def:modified-Hr}
    \|h\|_{mH^{r,k},\bar b}^2 &= \int_{\X}  h^2W_k\dd\mu \nonumber\\ &\qquad + \int_{\X} \sum_{p=1}^r \po \sum_{i=0}^{p-1}\omega_{i,p} |\na_x^i\na_y^{p-i} h |^2  + \omega_{p,p} |\po\na_x^p-\xi\na_x^{p-1}\na_y\pf h|^2 \pf \po 1 + \varepsilon_p W_{k-p} \pf \dd \mu,
\end{align}
where $(\omega_{i,p})$ and $(\varepsilon_p)$ are two families of positive parameters to be fixed later, and $\xi = \frac{\gamma}{2} + \sqrt{\frac{\gamma^2}{4}+1}$. \lj{The exact value of $\xi$ does not matter, as the important term is the scalar product $\na_x^p\cdot \na_x^{p-1}\na_y$. However this expression allows for simplified computation, see inequality~\eqref{eq:low_bound_Gamma}}. For the sequel of this work, we will drop the dependence in $\bar b$ in the definition of the norm and simply write $\|\cdot\|_{mH^{r,k}}$. Write:
\[mH^k=\left\{h\in \mathcal L^2(\X,\R),\quad  \|h\|_{mH^{k,k}}<\infty \right\}.\]
The goal is to show that for $f\in\mathcal P(U)$, we may chose $\bar b$ such that $\|P_tf-\mu(f)\|_{mH^{r,k}}$ decreases  exponentially fast along the dynamic.

\begin{thm}\label{thm:Hr-hypo}
Under Assumptions~\ref{assu:basic},~\ref{assu:growth} and~\ref{assu:Hr-setting}, for all $f\in\mathcal P(U)$, $k\in\N$, there exists $0 < \bar b < \beta/(2d)$ such that for all $t\geqslant 0$, $P_tf\in mH^k$, and we have that there is $q_k>0$ such that:
\begin{equation}
    \|P_tf\|_{mH^{k,k}} \leqslant e^{-q_kt}\|f\|_{mH^{k,k}}.
\end{equation}
\end{thm}

This section will be mostly about the demonstration of this theorem, which relies on Gamma calculus. We start with a quick review of Gamma calculus in Section~\ref{s-sec:Gamma_Calculus} for non-singular potential, then we will prove $H^1$ convergence in Section~\ref{s-sec:H1-hypo}, which serves as an initial case for an induction argument which will prove Theorem~\ref{thm:Hr-hypo} in Section~\ref{s-sec:Hr-hypo}. Finally, Theorem~\ref{thm:estimates} will be proven is Section~\ref{s-sec:proof_estimates}.

In the sequel, we denote by $L$ the generator of the process~\eqref{eq:langevin} given by
\begin{equation}\label{eq:generateur}
    L = y\cdot\na_x  - \na U\cdot\na_y  - \gamma y\cdot\na_y  + \gamma \beta^{-1}\Delta_y,
\end{equation}
and by $L^*$ its adjoint in $L^2(\mu)$, given by:
\[
L^* = -y\cdot\na_x + \na U\cdot\na_y  - \gamma y\cdot\na_y  + \gamma \beta^{-1}\Delta_y.
\]

\subsection{Gamma calculus}\label{s-sec:Gamma_Calculus}

For smooth and bounded function $f:\mathcal X\to\R$, H\"ormander's theorem yields that $P_tf$ is smooth  and solves the following Kolmogorov equation:
\begin{equation}
    \left\{\begin{aligned}
    &\partial_t g = Lg,\\ &g(t=0,\cdot)=f,
    \end{aligned} \right.
\end{equation}
where $L$ is the generator of the process defined in~\eqref{eq:generateur}, see \cite[Proposition 2.5]{ergodicity}. We are interested in the evolution of quantities of the form:
\[
\int_{\mathcal X} \phi\left( P_tf\right)(z) \mu(\dd z)
\]
for quadratic functional $\phi:\mathcal C^{\infty}(\mathcal X)\to \mathcal C^{\infty}(\mathcal X)$ of the form:
\[
\phi(h) = |A\na^{\alpha} h|^2 = \sum_{\alpha_1 + \dots + \alpha_{2d} =\alpha} |A_{\alpha_1,\dots,\alpha_{2d}}\partial^{\alpha_1}_{z_1}\dots\partial^{\alpha_{2d}}_{z_{2d}}h|^2,
\]
for some $\alpha \in\N$ and tensor $A$.
To this end, define:
\begin{equation}\label{def:Gamma}
    \Gamma_{L,\phi}(h) = \frac{1}{2}\left( L(\phi(h)) - D_h\phi(h)Lh\right),
\end{equation}
where $D_h\phi$ denote the differential operator of $\phi$. This definition is motivated by the following formal computation:
\begin{equation}\label{eq:derivative}
    \frac{\dd }{\dd t} \int_{\mathcal X} \phi\left( P_tf(z)\right)  \mu(\dd z) = -\int_{\mathcal X} \Gamma_{L,\phi}\left( P_tf(z)\right)\mu(\dd z),
\end{equation}
which uses that $\mu$ is invariant for $P_t$. If we can get an inequality of the form 
\[
c\int_{\mathcal X} \phi(h)(z) \mu(\dd z) \leqslant  \int_{\mathcal X}  \Gamma_{L,\phi}(h)(z) \mu(\dd z)
\]
for all smooth $h$, then this implies the convergence:
\begin{equation}\label{eq:conv_general}
\int_{\mathcal X} \phi\left( P_tf(z)\right) \dd z \leqslant e^{-c t}\int_{\mathcal X} \phi\left( f(z)\right) \dd z.
\end{equation}
\lj{Gamma calculus} is a tool particularly well adapted for the study of quadratic functional, because of the following identity:

\begin{prop}\label{prop:calcul_Gamma}
If there exists $A=(A_1,\cdots,A_p):\mathcal C^{\infty}\to (\mathcal C^{\infty})^p$ a linear operator such that $\phi(h) = |Ah|^2$, then
\[
\Gamma_{L,\phi}(h) = \Gamma_{L,2}(Ah) + Ah\cdot [L,A]h,
\]
where $\Gamma_{L,2}(Ah)=\sum_{i=1}^p\Gamma(A_ih)$, $\Gamma(h)=\Gamma_{L,h^2}(h)$, and $[L,A]=\left([L,A_1],\cdots,[L,A_p] \right)$.
\end{prop}

In the non-singular case, we can use this to show $H^1(\mu)$ (and even $H^k(\mu)$) convergence of $P_t(f)$ towards $\mu(f)$: If there exists $0<m\leqslant M$ such that $mI_d \leqslant \na^2 U \leqslant MI_d$, where $I_d$ stands for the identity matrix of size $d$, then we can show that
\[
\phi(h) = ch^2 + |(\na_x + \na_y)h|^2 \lj{+ |\na_yh|^2},
\]
for some $c>0$, satisfies inequality~\eqref{eq:conv_general} thanks to a Poincar\'e inequality (see Proposition~\ref{prop:Poincare} below), and $\int_{\mathcal X} \phi(h)(z) \mu(\dd z)$ is equivalent to $H^1(\mu)$. For a more complete introduction, see~\cite{Gammacalculus}. In the singular case, $U$ is not convex and does not have a bounded Hessian, and a slightly more complex norm must be defined, as we will now see. Of course, one would also have to justify the derivative~\eqref{eq:derivative}. We will avoid this in the proof of Theorem~\ref{thm:estimates} thanks to the Lumer-Phillips theorem. \lj{Since, in the sequel of this work, $L$ will always denote the generator of the Langevin process, we will drop the dependency in $L$ to simplify the notation, and write $\Gamma_{\phi}$ instead of $\Gamma_{L,\phi}$.}

\subsection{\texorpdfstring{$H^1$}{H1}-setting}\label{s-sec:H1-hypo}

We now adapt the previous section to the singular case, as done in~\cite{BaudoinGordinaHerzog}, to serve as the initialization of our induction argument. However, the modified $H^1$-norm used in the cited work does not contain any Lyapunov function in the $\dot{H}^1$ part (i.e. the term with derivative of order 1), and we need some for the induction. Hence we use a slightly different norm, given in~\eqref{def:modified-Hr}, which becomes for $r=1$:
\begin{equation}\label{def:modified-H1}
    \|h\|_{mH^{1,k}} = \int_{\X} \po h^2 W_k + \omega\po |\na_y h|^2 + |(\na_x-\xi\na_y)h|^2\pf \po 1+\varepsilon W_{k-1}\pf  \pf\dd \mu  ,
\end{equation}
where $\xi=\frac{\gamma}{2} + \sqrt{\frac{\gamma^2}{4}+1}$, and $\varepsilon>0$, will be determined later. As in~\cite{BaudoinGordinaHerzog}, we need a local Poincar\'e inequality:

\begin{prop}\label{prop:Poincare}
For all compact set $K\subset\mathcal X$, $\mu$ satisfies a Poincar\'e inequality on $K$: there exists $\rho>0$ such that for all $f\in\mathcal C^{\infty}(\mathcal X,\R)$:
\begin{equation}\label{eq:poincare}
    \int_Kf^2\dd \mu \leqslant \rho \int_K |\na f|^2 \dd \mu + \frac{1}{\mu(K)}\left(\int_K f \dd \mu\right)^2.
\end{equation}
\end{prop}

\begin{proof}
Proof of such inequality can be found in~\cite{BakryGentilLedoux}.
\end{proof}

The functions $W$ and $V$ satisfy the following properties, which are described in \cite[Theorem 4.15]{BaudoinGordinaHerzog}:

\begin{prop}\label{prop:Lyapunov}
Under Assumptions~\ref{assu:basic} and~\ref{assu:growth}, for all $\bar b>0$, all $k\in\N$, there exists $V_{0,\bar b}:\X\to[1,\infty)$ such that if $V=V_{0,\bar b}^{1/k}$, $W_p=V^p + \lambda$ for some $\lambda>0$, $V:\X\to[1,\infty)$ is $\C^2$, and there exist $\alpha,\sigma>0$, $J\subset \X$ such that:
\begin{itemize}
    \item For all $\varepsilon>0$, there exists $C>0$ such that \begin{equation}\label{eq:order-lyapu}
        C^{-1}e^{(1-\varepsilon)\bar b H}<V<Ce^{(1+\varepsilon)\bar b H}.
    \end{equation}
    \item $J$ is compact, connected, and for all $p\in [1,k]$:
    \[
    L^*V^{p} \leqslant -\alpha \frac{p}{k}  V^{p} + \sigma \mathbbm{1}_J.
    \]
    \item For all $p\in\llbracket 1,k \rrbracket$, $g\in\C^{\infty}_b(\X,\R)$, we have that $Lg W_p,gL^*W_p\in L^1(\mu)$ and 
    \[
    \int_{\X}Lg W_p \dd\mu = \int_{\X}g L^* W_p \dd\mu.
    \]
    \item $\forall (x,y)\in\X$,
    \begin{equation}\label{eq:lower-bound-lyapu}
        V^k(x,y) \geqslant \frac{2\sigma\mu(J^c)}{\alpha\mu(J)}.
    \end{equation}
    \item $\forall (x,y)\in\X, v\in\R^d$,
    \begin{equation}\label{eq:control_R}
    W_k(x,y)|v|^2\geqslant \beta(\sigma\rho'+1) \po  \po \frac{1}{2} + \frac{2}{\gamma^2} \pf|v|^2 + \frac{1}{2\gamma^2}|\na^2 U v|^2 \pf,
    \end{equation}
    where 
    \begin{equation}\label{eq:rho-prime}
        \rho' = 4(1 + \xi^2)\rho/\gamma,
    \end{equation}
    and $\rho$ is the constant from the Poincar\'e inequality on $J$.
\end{itemize}
\end{prop}
\lj{In particular, \eqref{eq:order-lyapu} yields that $\lim_{H\rightarrow \infty}V =\infty$, and \eqref{eq:lower-bound-lyapu} that the set $J$ is not so large.} 
In order to study the evolution of the norm~\eqref{def:modified-H1} along the trajectories, we use Gamma calculus. \lj{All computations will be performed on functions $h:\X\to\R$ such that
\begin{equation}\label{eq:mean-zero}
\int_{\X}h \dd \mu = 0.
\end{equation}}
Recall the definition~\eqref{def:Gamma} of $\Gamma_{\phi}$. We have
\[
[L,\na_x] = \na^2U \na_y,\qquad [L,\na_y] = -\na_x + \gamma \na_y.
\]
Hence from Proposition~\ref{prop:calcul_Gamma} we get that:
\begin{equation*}
    \Gamma_{|\na_y \cdot|^2}(h) \geqslant  \gamma\beta^{-1}|\na_y^2 h|^2 - \na_y h\cdot(\na_x-\xi\na_y)h + (\gamma-\xi)|\na_yh|^2,
\end{equation*}
and:
\begin{align*}
\Gamma_{|(\na_x - \xi \na_y)\cdot|^2}(h) &\geqslant \gamma\beta^{-1}|(\na_x\na_y-\xi\na_y^2)h|^2 + \xi|(\na_x - \xi\na_y)h|^2 \\ &\quad+ \xi(\xi-\gamma)\na_y h\cdot(\na_x-\xi\na_y)h + \na^2U\na_y h\cdot(\na_x-\xi\na_y)h.
\end{align*}
Finally, with our choice of $\xi$ and $\phi(h) = |\na_y h|^2 + |(\na_x - \xi \na_y)h|^2$, one gets:
\begin{align*}
    \Gamma_{\phi}(h) &= \Gamma_{|\na_y \cdot|^2}(h) + \Gamma_{|(\na_x - \xi \na_y)\cdot|^2}(h)  \\&\geqslant  \gamma\beta^{-1}|\na_y^2 h|^2 + \gamma\beta^{-1}|(\na_x\na_y-\xi\na_y^2)h|^2 \\ &\quad+ \gamma |(\na_x - \xi\na_y)h|^2 - \frac{2}{\gamma}|\na_yh|^2 + \na^2U\na_y h\cdot(\na_x-\xi\na_y)h.
\end{align*}
Young inequality then yields:
\begin{equation}\label{eq:low_bound_Gamma}
    \Gamma_{\phi}(h) \geqslant  \gamma\beta^{-1}|\na_y^2 h|^2 + \gamma\beta^{-1}|(\na_x\na_y-\xi\na_y^2)h|^2 + \frac{\gamma}{2} |(\na_x - \xi\na_y)h|^2 - \mathcal{R}(x,\na_yh),
\end{equation}
where
\[
\mathcal R(x,v) = \frac{2}{\gamma}|v|^2 + \frac{|\na^2Uv|^2}{2\gamma}.
\]
Write:
\[
H(t) = \|P_tf-\int_{\mathcal X}f \dd\mu\|_{mH^{1,k}}.
\]
Because of the Lyapunov function, we have additional terms in the derivative of $H$ in comparison to the regular case. \lj{Using that $P_tf$ solves the Kolmogorov equation $\partial_tP_tf = LP_tf$, this formally reads
\begin{multline*}
H'(t) = \int_{\X}  P_tfLP_tf W_k \dd \mu \\+ \omega  \int_{\X}\po \na_yP_tf \cdot \na_y LP_tf + (\na_x - \xi\na_y)P_tf \cdot (\na_x - \xi\na_y)LP_tf \pf \po 1+\varepsilon W_{k-1}\pf  \dd \mu.
\end{multline*}

\begin{lem}\label{lem:formal-derivative-H1}
For all $h\in\mathcal C^{\infty}_b(\X,\R)$
\begin{multline*}
\int_{\X}  hLh W_k \dd \mu + \omega  \int_{\X}\po \na_yh \cdot \na_y Lh + (\na_x - \xi\na_y)h \cdot (\na_x - \xi\na_y)Lh \pf \po 1+\varepsilon W_{k-1}\pf  \dd \mu \\= \int_{\X} \po h^2L^*W_k - \gamma\beta^{-1}|\na_yh|^2W_k \pf\dd \mu + \omega\varepsilon  \int_{\X}\po |\na_yh|^2 + |(\na_x - \xi\na_y)h|^2 \pf L^* W_{k-1} \dd \mu \\ - \omega \int_{\X} \po \Gamma_{|\na_y\cdot|^2}(h)  + \Gamma_{|(\na_x - \xi\na_y)\cdot|^2}(h)\pf \po 1+\varepsilon W_{k-1}\pf  \dd \mu.
\end{multline*}
\end{lem}

\begin{proof}
The fact that $\mu$ is an invariant measure of the Langevin process yields that
\begin{equation*}
0 = \int_{\X}  L\po h^2 W_k \pf \dd \mu =  \int_{\X}  L\po h^2\pf W_k  \dd \mu +  \int_{\X} h^2f L W_k \dd \mu + 2\int_{\X}   \Gamma\po h^2f, W_k \pf \dd \mu.
\end{equation*}
This implies that 
\begin{multline*}
    2\int_{\X}  \Gamma\po h^2f, W_k \pf \dd \mu = -\int_{\X}  L\po h^2\pf W_k  \dd \mu -  \int_{\X} h^2f L W_k \dd \mu \\= -  \int_{\X} h^2f L^* W_k \dd \mu -  \int_{\X} h^2f L W_k \dd \mu,
\end{multline*}
so that
\[
0 = \int_{\X}  L\po h^2\pf W_k  \dd \mu - \int_{\X} h^2f L^* W_k \dd \mu.
\]
This equality allows us to write
\begin{multline*}
\int_{\X}  hLh W_k \dd \mu = \int_{\X}  hLh W_k \dd \mu - \int_{\X}  L\po h^2\pf W_k  \dd \mu + \int_{\X} h^2f L^* W_k \dd \mu \\ = -\int_{\X}  \Gamma\po h^2\pf W_k  \dd \mu + \int_{\X} h^2f L^* W_k \dd \mu.
\end{multline*}
The same computation yields
\begin{multline*}
 \int_{\X} \na_yh \cdot \na_y Lh \po 1+\varepsilon W_{k-1}\pf \dd \mu  = - \int_{\X} \Gamma_{|\na_y\cdot|^2}\po h \pf  \po 1+\varepsilon W_{k-1}\pf \dd \mu \\+ \int_{\X} |\na_yh|^2 L^* \po 1+\varepsilon W_{k-1}\pf \dd \mu,
\end{multline*}
and
\begin{multline*}
 \int_{\X} (\na_x - \xi\na_y)h \cdot (\na_x - \xi\na_y) Lh \po 1+\varepsilon W_{k-1}\pf \dd \mu  = - \int_{\X} \Gamma_{|(\na_x - \xi\na_y)\cdot|^2}\po h \pf  \po 1+\varepsilon W_{k-1}\pf \dd \mu \\+ \int_{\X} |(\na_x - \xi\na_y)h|^2 L^* \po 1+\varepsilon W_{k-1}\pf \dd \mu,
\end{multline*}
which concludes the proof.
\end{proof}
}

Hence to conclude in the $H^1$ setting, we need a control on the additional terms. This is the object of the next lemma, which is also the initialisation of the induction argument used in the next section for the $H^k$ setting:
\begin{lem}\label{lem:H1-hypo}
For all $k\in\N^*$, under Assumptions~\ref{assu:basic} and~\ref{assu:growth}, there exist $q,\omega,\varepsilon>0$ such that:
\begin{multline}\label{eq:H1-hypo}
\int_{\X} \po h^2L^*W_k - \gamma\beta^{-1}|\na_yh|^2W_k \pf\dd \mu \\  + \omega\varepsilon  \int_{\X}\po |\na_yh|^2 + |(\na_x - \xi\na_y)h|^2 \pf L^* W_{k-1} \dd \mu \\ - \omega \int_{\X} \po \Gamma_{|\na_y\cdot|^2}(h)  + \Gamma_{|(\na_x - \xi\na_y)\cdot|^2}(h)\pf \po 1+\varepsilon W_{k-1}\pf  \dd \mu \\ \leqslant -q\po \|h\|_{mH^{1,k}} + \int_{\X}  W_{k-1} \po |\na_y^2 h|^2 + |\na_x\na_yh|^2 \pf \dd\mu\pf,
\end{multline}
for all $h\in \C^{\infty}_b (\X,\R)$.
\end{lem}

Notice that in contrast to inequality (4.8) of~\cite{BaudoinGordinaHerzog}, we keep the higher order derivative for the induction, see Section~\ref{s-sec:Hr-hypo}. 

\begin{proof}
We first treat the $L^2$-term using the Lyapunov property of $V$:
\[
\int_{\X} h^2L^*W_k \dd \mu \leqslant -\frac{\alpha}{1+\lambda} \int_{\X} h^2 W_k \dd \mu + \sigma \int_{J} h^2 \dd \mu.
\]
\lj{Using Cauchy-Schwarz inequality, inequality~\eqref{eq:lower-bound-lyapu} and that $h$ satisfies~\eqref{eq:mean-zero}, one gets
\[
\po \int_{J} h \dd\mu \pf^2 = \po \int_{J^c} h \dd\mu \pf^2 \leqslant \mu(J^c) \int_{\mathcal X} h^2 \dd\mu \leqslant \frac{\alpha \mu(J)}{2\sigma(1+\lambda)} \int_{\mathcal X} h^2 W_k \dd\mu.
\]
The local Poincar\'e inequality~\eqref{eq:poincare} on $J$ can then be written as follow:
\begin{align*}
\sigma \int_{J} h^2 \dd \mu &\leqslant  \sigma \rho'\frac{\gamma}{2}\int_{J}(|(\na_x-\xi\na_y)h|^2 + |\na_yh|^2)\dd \mu + \frac{\sigma}{\mu(J)}\po\int_{J} h \dd\mu \pf^2 \\ & \leqslant \sigma\rho'\int_{\X} \frac{\gamma}{2}|(\na_x-\xi\na_y)h|^2 - \mathcal R(x,\na_yh) + \mathcal R(x,\na_yh) + |\na_yh|^2\dd \mu \\ & \qquad + \frac{\alpha}{2(1+\lambda)} \int_{\X} h^2 W_k \dd \mu,
\end{align*}
where $\rho'$ was defined in~\eqref{eq:rho-prime}.}
Using inequality~\eqref{eq:control_R}, we get:
\[
\int_{\X} \mathcal R(x,\na_yh) + \frac{\gamma}{2}|\na_yh|^2\dd \mu \leqslant \frac{\gamma\beta^{-1}}{\sigma\rho'+1}\int_{\X} |\na_yh|^2W_k \dd \mu.
\]
We also have from inequality~\eqref{eq:low_bound_Gamma}:
\[
\int_{\X} \frac{\gamma}{2}|(\na_x-c\na_y)h|^2 - \mathcal R(x,\na_yh) \dd \mu \leqslant \int \Gamma_{|\na_y\cdot|^2}(h) + \Gamma_{|(\na_x - c\na_y)\cdot|^2}(h) \dd \mu. 
\]
We treat the second line of~\eqref{eq:H1-hypo} using the Lyapunov property as follow:
\begin{align*}
    \int_{\X}& \po |\na_y h|^2L^*W_{k-1} + |(\na_x-\xi\na_y)h|^2L^*W_{k-1} \pf \dd \mu \\ &\leqslant \int_{\X} \po |\na_y h|^2 + |(\na_x-\xi\na_y)h|^2 \pf \sigma\mathbbm{1}_J\dd \mu \\ &\leqslant \frac{2\sigma}{\gamma} \int_{\X} \po \frac{\gamma}{2}|\na_y h|^2 + \mathcal R(x,\na_y h) + \frac{\gamma}{2} |(\na_x-\xi\na_y)h|^2 - \mathcal R(x,\na_y h) \pf \dd\mu \\ &\leqslant  \frac{2\sigma}{\beta(1+\sigma\rho')}\int_{\X} |\na_y h|^2W_k \dd\mu + \frac{2\sigma}{\gamma} \int \Gamma_{|\na_y\cdot|^2}(h) + \Gamma_{|(\na_x - \xi\na_y)\cdot|^2}(h) \dd \mu.
\end{align*}
Next, the $\varepsilon$-term of the third line is bounded thanks to~\eqref{eq:low_bound_Gamma} as follow:
\begin{multline*}
    - \int_{\X} \po \Gamma_{|\na_y\cdot|^2}(h)  + \Gamma_{|(\na_x - \xi\na_y)\cdot|^2}(h)\pf  W_{k-1}  \dd \mu \\ \leqslant -\gamma \beta^{-1} \int_{\X} \po |\na_y^2 h|^2 + |(\na_x\na_y-\xi\na_y)h|^2 \pf W_{k-1} \dd \mu \\ - \frac{\gamma}{2}\int_{\X} |(\na_x-\xi\na_y)h|^2W_{k-1} \dd\mu + \int_{\X} \mathcal R(x,\na_y h) W_{k-1} \dd\mu. 
\end{multline*}
Thanks to Assumption~\ref{assu:growth} we get that there is $C>0$ such that $|\na^2U|\leqslant CV$. This yields that
\begin{multline*}
\mathcal R(x,\na_y h) W_{k-1} \leqslant C' \po |\na_yh|^2 + |\na^2U\na_yh|^2 \pf W_{k-1} \\\leqslant C'' \po |\na_yh|^2 + |\na_yh|^2 V^{1/(k-1)}  \pf W_{k-1} \leqslant \kappa |\na_yh|^2 W_k
\end{multline*}
and:
\begin{multline*}
- \int_{\X} \po \Gamma_{|\na_y\cdot|^2}(h)  + \Gamma_{|(\na_x - \xi\na_y)\cdot|^2}(h)\pf  W_{k-1}  \dd \mu  \\ \leqslant -\gamma \beta^{-1} \int_{\X} \po |\na_y^2 h|^2 + |\na_x\na_yh|^2 \pf W_{k-1} \dd \mu  \\ - \frac{\gamma}{2}\int_{\X} |(\na_x-\xi\na_y)h|^2W_{k-1} \dd\mu + \kappa \int_{\X} |\na_y h|^2 W_k \dd \mu.
\end{multline*}
We then get the following upper bound for the left-hand side of equation~\eqref{eq:H1-hypo}:
\begin{multline*}
    -\frac{\alpha}{2(1+\lambda)} \int_{\X} h^2 W_k \dd \mu  \\ -\po \gamma\beta^{-1} -  \frac{2\varepsilon\sigma\beta^{-1}}{\sigma\rho'+1} - \kappa\varepsilon\omega\pf\int_{\X} |\na_yh|^2 W_k \dd \mu \\  -\varepsilon\omega\gamma \beta^{-1} \int_{\X} \po |\na_y^2 h|^2 + |\na_x\na_yh|^2 \pf W_{k-1} \dd \mu  \\ + \po \frac{2\sigma\varepsilon\omega}{\gamma} + \sigma\rho' - \omega \pf \int_{\X} \Gamma_{|\na_y\cdot|^2}(h)  + \Gamma_{|(\na_x - \xi\na_y)\cdot|^2}(h) \dd\mu \\  -\frac{\gamma\varepsilon\omega}{2}\int_{\X} |(\na_x-\xi\na_y)h|^2W_{k-1} \dd\mu.
\end{multline*}
Taking $\omega > \sigma\rho'$, $\varepsilon$ small enough so that the second and fourth terms of the previous bounds are negative, and using the fact that $W_{k-1}\leqslant W_k$ (because $V\geqslant 1$) then concludes the proof.
\end{proof}
Removing the second order terms, \lj{Lemma~\ref{lem:formal-derivative-H1} and~\ref{lem:H1-hypo}} formally yields:
\begin{align*}
    H'(t) \leqslant & -qH(t),
\end{align*}
for some $q>0$. Since $W_k\geqslant 1$ for all $k>0$, this would imply convergence in $H^1$-norm
\[\|P_tf-\mu(f)\|_{H^1}\leqslant Ce^{-qt}\|f-\mu(f)\|_{mH^{1,k}},\]
for some $C>0$. This is only a formal computation, which we will justify in the proof of Theorem~\ref{thm:Hr-hypo}.

\subsection{\texorpdfstring{$H^k$}{Hr}-setting}\label{s-sec:Hr-hypo}

Building upon the result of the previous section, we may now perform computations in higher-order Sobolev spaces. From those computations arise derivatives of $U$ of arbitrary orders, hence we now need to assume Assumption~\ref{assu:Hr-setting}. Since the Lyapunov function satisfies $V=e^{bH+o(H)}$, for some $0<b<\beta$, this assumption implies: for all $\alpha\in\N^d$, there exists $\kappa_{\alpha}>0$ such that
\begin{equation}\label{eq:control-derivative}
    |\partial^{\alpha} U | \leqslant \kappa_{\alpha} V.
\end{equation}
Recall the definition of the modified Sobolev norm from~\eqref{def:modified-Hr}, fix some $r\in\llbracket 1,k\rrbracket$, and write:
\[
H_{r,k}(t) = \|P_tf-\int_{\X} f\dd \mu \|_{mH^{r,k}}.
\]
\lj{The same proof as Lemma~\ref{lem:formal-derivative-H1} yields the following formal derivative}:
\begin{align}\label{eq:derivativeHk}
   H_{r,k}'(t) &=  \int_{\X} \po (P_tf)^2L^*W_k -   \gamma\beta^{-1} |\na_yP_tf|^2 W_k \pf \dd\mu \nonumber\\ &\qquad + \int_{\X} \sum_{p=1}^r \varepsilon_p \po \sum_{i=0}^{p-1} \omega_{i,p}|\na_x^i\na_y^{p-i} P_tf |^2 + \omega_p  |\po\na_x^p-\xi\na_x^{p-1}\na_y\pf P_tf|^2 \pf L^*W_{k-p} \dd \mu \nonumber\\ &\qquad
   - \int_{\X} \sum_{p=1}^r \po\sum_{i=0}^{p-1}  \omega_{i,p}\Gamma_{i,p-i}(P_tf) + \omega_p \Gamma_p(P_tf) \pf \po 1+\varepsilon_p W_{k-p} \pf\dd\mu
\end{align}
where we wrote:
\[
\Gamma_{l,p} = \Gamma_{\left|\na_x^l\na_y^p\cdot\right|^2},
\]
and 
\[
\Gamma_p = \Gamma_{\left|\po\na_x^p-\xi\na_x^{p-1}\na_y\pf\cdot\right|^2}
\]
for any $l,p\in \N\times\N^*$.
As in the previous section, we need to bound this derivative. This bound is obtained through an induction principle on $r\in \llbracket1,k\rrbracket$, and the repetitive use of Proposition~\ref{prop:calcul_Gamma}. First we compute:
\[
[L,\na_x^l\na_y^p] = \gamma \na_x^l\na_y^p - \mathbbm{1}_{p\geqslant 1}\na_x^{l+1}\na_y^{p-1} + \sum_{i=1}^{l} \begin{vmatrix} l\\ i\end{vmatrix} \na^{i+1} U\otimes \na_x^{l-i}\na_y^p, 
\]
where the terms in the sum as to be understood as:
\[
\begin{vmatrix} l\\ i\end{vmatrix} \na^{i+1} U \otimes \na_x^{l-i}\na_y^p = \left( \sum_{j=1}^d\sum_{k_s\leqslant l_s} \prod_{s=1}^d \binom{l_s}{k_s} \partial_{x_j}\partial^{k}U \partial_{y_j} \partial_x^{l-k}\partial_y^{p+1}  \right)_{\sum l_j=l,\sum p_j =p}.
\]
Using Proposition~\ref{prop:calcul_Gamma}, we get that:
\begin{align*}
    \Gamma_{l,p}(h) &= \Gamma(\na_x^l\na_y^ph) + \na_x^l\na_y^ph \cdot [L,\na_x^l\na_y^p]h.
\end{align*}
If $p\geqslant 2$, we may then bound below using inequality~\eqref{eq:control-derivative}:
\begin{equation}\label{eq:low_bound_Gamma_l_p}
    \Gamma_{l,p}(h) \geqslant \gamma\beta^{-1}|\na_x^l\na_y^{p+1}h|^2  -\theta_{l,p} \po |\na_x^l\na_y^ph|^2 + V \sum_{i=1}^l|\na_x^{l-i}\na_y^{p+1}h|^2 \pf - \frac{1}{4\gamma}|\na_x^{l+1}\na_y^{p-1}|^2,
\end{equation}
for some $\theta_{l,p}\geqslant 0$. If $p=1$, there is a term $|\na_x^{l+1}h|^2$ in the commutator, which will be problematic, as we cannot control it by induction using the derivative of the lower-order terms. In order to take care of it, we fix a small parameter $\eta_l>0$, use the inequality $a\cdot b \leqslant |a|^2/(2\eta) + \eta|b|^2/2$ and we bound below again using inequality~\eqref{eq:control-derivative}:
\begin{multline}\label{eq:low_bound_Gamma_l_1}
    \Gamma_{l,1}(h) \geqslant \gamma\beta^{-1}|\na_x^l\na_y^{2}h|^2  - \theta_{l,1} \po |\na_x^l\na_y^1h|^2 + V\sum_{i=1}^l|\na_x^{l-i}\na_y^{2}h|^2 \pf - \eta_l|(\na_x^{l+1}-\xi\na_x^l\na_y)h|^2.
\end{multline}
As usual in hypocoercive computations for kinetic processes, the derivative of a Sobolev norm of order $p$ will lack in its derivative a term of the form $|\na_x^ph|^2$. It will be the derivative of the cross term $\na_x^ph.\na^{p-1}_x\na_y$ in $|\left(\na_x^p-c\na_x^{p-1}\na_y\right)h|^2$ that will give us this missing derivative. Indeed, using that $\xi\geqslant \gamma$ and inequality~\eqref{eq:control-derivative}, one has:
\begin{multline}
\na_x^ph\cdot [L,-\xi \na_x^{p-1}\na_y]h = \xi \na_x^ph\cdot \po \na_x^p - \gamma \na_x^{p-1}\na_y -  \sum_{i=1}^{p} \begin{vmatrix} p\\ i\end{vmatrix} \na^{i+1} U\otimes \na_x^{p-i}\na_y  \pf h \\ \geqslant \frac{\gamma}{2} |\na_x^p h|^2 - C V \po  \sum_{l=1}^p |\na_x^{p-l}\na_yh|^2 + \sum_{l=1}^{p-1} |\na_x^{p-1-l}\na_y^{2}h|^2 \pf.
\end{multline}
We also have that
\[
|\po \na_y\na_x^p-\xi\na_x^{p-1}\na_y^2\pf h|^2 \geqslant \frac{1}{2} |\na_x^{p}\na_y h|^2 - \frac{1}{2}\xi |\na_x^{p-1}\na_y^2 h|^2
\]
Hence after rearranging the terms, Proposition~\ref{prop:calcul_Gamma} again yields:
\begin{multline}\label{eq:low_bound_Gamma_p_seul}
    \Gamma_p(h) \geqslant \\ \frac{\gamma\beta^{-1}}{2}|\na_x^{p}\na_y h|^2 +  \frac{\gamma}{2}|\po\na_x^p-\xi\na_x^{p-1}\na_y\pf h|^2 - \theta_{p,0} V \po  \sum_{l=1}^p |\na_x^{p-l}\na_yh|^2 + \sum_{l=0}^{p-1} |\na_x^{p-1-l}\na_y^{2}h|^2 \pf , 
\end{multline}
for some $\theta_{p,0}>0$. Thanks to those different lower bounds, we will be able to prove the next lemma which is the central part of this section.
\begin{lem}\label{lem:Hr-hypo}
Fix $k\in\N^*$. There exists $(\omega_{i,p})_{1\leqslant i\leqslant p\leqslant k}$, $(\varepsilon_p)_{1\leqslant p\leqslant k}$, all positive, such that for all $r\in\llbracket1,k\rrbracket$, there exists $q_r>0$ such that for all $h\in\C_b^{\infty}(\mathcal X,\R)$:
\begin{align}\label{eq:hypo-rec}
    &\int_{\X} \po h^2L^*W_k -   \gamma\beta^{-1} |\na_yh|^2 W_k \pf \dd\mu \nonumber \\ &\qquad + \int_{\X} \sum_{p=1}^r \varepsilon_p \po \sum_{i=0}^{p-1} \omega_{i,p}|\na_x^i\na_y^{p-i} h |^2 + \omega_p  |\po\na_x^p-\xi\na_x^{p-1}\na_y\pf h|^2 \pf L^*W_{k-p} \dd \mu \nonumber \\ &\qquad
   - \int_{\X} \sum_{p=1}^r \po\sum_{i=0}^{p-1}  \omega_{i,p}\Gamma_{i,p-i}(h) + \omega_p \Gamma_p(h) \pf \po 1+\varepsilon_p W_{k-p} \pf\dd\mu \nonumber \\ &\leqslant -q_r\po \|h\|_{mH^{r,k}} + \int_{\X} \sum_{i=0}^r |\na_x^i\na_y^{r+1-i} h |^2 W_{k-r} \dd \mu \pf.
\end{align}
\end{lem}

\begin{proof}
The proof is done by induction. The case $r=1$ corresponds to Lemma~\ref{lem:H1-hypo}. Let $r\in\llbracket1,k\rrbracket$, and we suppose the inequality~\eqref{eq:hypo-rec} true for such an $r$. To prove the result at rank $r+1$, we want to bound:
\begin{align*}
    &-q_r\int_{\X} \po h^2W_k + \sum_{1\leqslant i+j\leqslant r,j\geqslant 1} |\na_x^i\na_y^{j} h |^2 W_{k-(i+j)} + \sum_{i=1}^r |\po\na_x^i-\xi\na_x^{i-1}\na_y\pf h |^2  W_{k-i}  \pf \dd\mu  \\ &\quad -q_r \int_{\X} \sum_{i=0}^r |\na_x^i\na_y^{r+1-i} h |^2 W_{k-r} \dd \mu \\ &\quad + \varepsilon_{r+1}\int_{\X} \po \sum_{i=0}^{r} \omega_{i,r+1}|\na_x^i\na_y^{r+1-i} h |^2 + \omega_{r+1}  |\po\na_x^{r+1}-\xi\na_x^{r}\na_y\pf h|^2 \pf L^*W_{k-(r+1)} \dd \mu \\ &\quad
   - \int_{\X} \po\sum_{i=0}^{r}  \omega_{i,r+1}\Gamma_{i,r+1-i}(h) + \omega_{r+1} \Gamma_{r+1}(h) \pf \po 1+\varepsilon_{r+1} W_{k-(r+1)} \pf\dd\mu.
\end{align*}
First we bound using the Lyapunov property:
\[
|\na_x^i\na_y^{r+1-i} h |^2 L^*W_{k-(r+1)} \leqslant \sigma |\na_x^i\na_y^{r+1-i} h |^2.
\]
The goal now is to check that for any order of derivative, we may chose $\varepsilon_{r+1}$, $(\omega_{i,r+1})$ and $(\omega_{r+1})$ such that inequality~\eqref{eq:hypo-rec} holds true with $(r+1)$ instead of $r$. To this end, we will use the inequality~\eqref{eq:low_bound_Gamma_l_p},\eqref{eq:low_bound_Gamma_l_1} and~\eqref{eq:low_bound_Gamma_p_seul} on the $\Gamma$'s. The term of order $r+2$ of derivatives is:
\begin{equation*}
-\gamma\beta^{-1}\int_{\X} \po \sum_{i=0}^r\omega_{i,r+1} |\na_x^i\na_y^{r+2-i}h|^2 + \frac{\omega_{r+1}}{2} |(\na_x^{r+1}\na_y - \xi\na_x^r\na_y^2)h|^2 (1+\varepsilon_{r+1}W_{k-(r+1)} ) \pf \dd\mu,
\end{equation*}
which is already what we were looking for. This does not impose any condition on the coefficient, except the positivity of the $(\omega)$. The derivative of order $r+1$ in $x$ which was missing from the rank $r$, is:
\begin{align*}
    \int_{\X} \po\varepsilon_{r+1}\omega_{r+1}\sigma - \frac{\gamma\omega_{r+1}}{2}(1+\varepsilon_{r+1}W_{k-(r+1)}) + \eta_r\omega_{r,r+1}  (1+\varepsilon_{r+1}W_{k-(r+1)})\pf  |\po\na_x^{r+1}-\xi\na_x^r\na_y\pf h|^2  \dd\mu.
\end{align*}
We choose $\eta_r = \frac{\gamma\omega_{r+1}}{8\omega_{r,r+1}}$, and $0<\varepsilon_{r+1}<\frac{\gamma}{8\sigma}$, so that we are left with:
\[
-\frac{\gamma\omega_{r+1}}{4}\int_{\X} |\po\na_x^{r+1}-\xi\na_x^r\na_y\pf h|^2\po1+\varepsilon_{r+1}W_{k-(r+1)}\pf  \dd\mu.
\]
For the other terms of order $r+1$, fix some $i\in\llbracket 0,r \rrbracket$. The term in $|\na_x^i\na_y^{r+1-i}h|^2$ is:
\begin{multline*}
    \int_{\X} \po -q_rW_{k-r}  + \sigma\varepsilon_{r+1}\omega_{i,r+1} \right. \\ \left. + \po \omega_{i,r+1}\theta_{i,r+1-i} + \omega_{i+1,r+1}\po 1/(4\gamma) + \theta_{i+1,r+1-(i+1)}V \pf \pf  (1+\varepsilon_{r+1} W_{k-(r+1)}) \pf |\na_x^i\na_y^jh|^2 \dd\mu.
\end{multline*}
Using that $VW_{k-(r+1)}\leqslant W_{k-r}(1+\lambda)$, we may choose the $(\omega_{i,r+1})$ small enough so that this last quantity is less than:
\[
-\frac{q_r}{2}\int_{\X} W_{k-r}|\na_x^i\na_y^jh|^2 \dd\mu.
\]
The lower order terms are treated in the same way, and this concludes the induction and the proof.
\end{proof}

\begin{proof}[Proof of Theorem~\ref{thm:Hr-hypo}]
Instead of trying to justify the derivative~\eqref{eq:derivativeHk}, we apply a semi-group argument. Our goal is to apply Lumer-Phillips theorem to the operator $L+q_k/2I$, where $I$ denotes the identity operator of $mH^k$. This theorem can be stated as follows: an operator $A$ on a Hilbert space generates a contraction semi-group if and only if it is maximally dissipative, see~\cite[Chapter IX, p.250]{Func_anal}. Fix $0<\bar b < \beta/(2d)$ and $k\in\N$. $mH^k$ is a Hilbert space, with scalar product
\begin{multline*}
\left\langle f,g \right\rangle_{mH^{k}}^2 = \int_{\X}  fg W_k\dd\mu \\ + \int_{\X} \sum_{p=1}^k \po \sum_{i=0}^{p-1}\omega_{i,p} \na_x^i\na_y^{p-i} f \cdot  \na_x^i\na_y^{p-i} g + \omega_{p,p} \po\na_x^p-\xi\na_x^{p-1}\na_y\pf f \cdot \po\na_x^p-\xi\na_x^{p-1}\na_y\pf g \pf \\ \po 1 + \varepsilon_p W_{k-p} \pf \dd \mu.
\end{multline*}
Lemma~\ref{lem:Hr-hypo} and a density argument yield that the operator $L+q_k/2I$, is dissipative:
\[
\forall h\in D(L),\, \left\langle(L+q_k/2I)h,h\right\rangle_{mH^k}\leqslant 0,
\]
where $D(L)$ denote the domain of $L$ \lj{in $mH^k$ defined by:
\[
f\in\D(L),\, g = Lf \Leftrightarrow f\in mH^k,\,  \lim_{t\rightarrow 0} \left\|\frac{P_tf-f}{t} - g\right\|_{mH^k} = 0.
\]
}We are left to show that $L+qI$, for some $q<q_k/2$, is surjective. Fix such a $q<q_k/2$. Thanks to Lemma~\ref{lem:Hr-hypo}, we have that \[\Lambda:(f,g)\mapsto \left\langle-(L+qI)f,g\right\rangle_{mH^k}\] is coercive and continuous from $\po mH^k\pf^2$ to $\R$, as long as $k\geqslant 2$. Hence, we may apply Lax-Milgram theorem to get that for all $g\in mH^k$, there exists $f\in mH^k$ such that for all $h\in mH^k$, we have: 
\[
\left\langle-(L+qI)f,h\right\rangle_{mH^k} = \left\langle-g,h\right\rangle_{mH^k},
\]
which implies that $f$ is a solution to the equation
\[
(L+qI)f = g,
\]
and $L+q_k/2I$ is maximally dissipative. Lumer-Phillips theorem then yields that the semi-group generated by $L+q_k/2$ is a contraction on $mH^k$: for all $f\in mH^k$
\[
\|e^{q_kt/2}P_tf\|_{mH^k} \leqslant \|f\|_{mH^k}.
\]
For all $f\in\mathcal P(U)$, we may fix $\bar b>0$ small enough so that $f\in mH^k$, and this concludes the proof.
\end{proof}

\subsection{Proof of Theorem~\ref{thm:estimates}}\label{s-sec:proof_estimates}

The proof of Theorem~\ref{thm:estimates}, based on Theorem~\ref{thm:Hr-hypo}, uses the Lyapunov structure of $V$.

\begin{proof}[Proof of theorem~\ref{thm:estimates}]
Fix some $f\in\mathcal P(U)$. Thanks to Proposition~\ref{prop:Lyapunov}, we may fix the parameter $\bar b$ from the Lyapunov function such that $\|f\|_{mH^k}<\infty$. For all $k\in\N$, we can apply Theorem~\ref{thm:Hr-hypo} to get some $C>0$ such that:
\[
\|P_tf\|_{mH^k} \leqslant Ce^{-q_k t}.
\]
Fix $0<b<\beta$, $C,\varepsilon>0$ such that $e^{(\beta-(1-\varepsilon)b)H}\leqslant CV_0$, for some $C>0$. Thanks to Assumption~\ref{assu:Hr-setting}, we may write for all $\alpha\in\N^{2d}$:
\begin{align*}
\int_{\mathcal X} \left| \partial^{\alpha} \left(P_tf e^{-bH} \right) \right|^2 &\leqslant C \sum_{|\sigma|\leqslant |\alpha|} \int_{\mathcal X} \left| \partial^{\sigma} P_tf  \right|^2\left| \partial^{\alpha-\sigma} e^{-bH}  \right|^2 \\ &\leqslant C'\sum_{|\sigma|\leqslant |\alpha|} \int_{\mathcal X} \left| \partial^{\sigma} P_tf  \right|^2 e^{-(1-\varepsilon)bH} \\ &= C'\sum_{|\sigma|\leqslant |\alpha|} \int_{\mathcal X} \left| \partial^{\sigma} P_tf  \right|^2 e^{(\beta-(1-\varepsilon)b)H}e^{-\beta H} \\ &\leqslant C''\|P_tf\|_{mH^{|\alpha|}} \\ &\leqslant C''e^{-q_{|\alpha|} t}.
\end{align*}
We conclude with Sobolev embedding.
\end{proof}

\section{Weak error expansion}\label{sec:scheme}

The main ingredient in the proof of a Talay-Tubaro expansion for a numerical scheme (or even simply weak convergence) are the estimates given in Theorem~\ref{thm:estimates}. It allows to control the error (in law) made by each step of the numerical scheme, uniformly in time. Hence we are now able to prove Theorem~\ref{thm:num_approx_finite_time} and Theorem~\ref{thm:num_approx_stationary}. This section is divided into two parts. The first one proves the existence of a Lyapunov function of order $e^{bH}$, for all $0<b<\beta$, as well as Theorem~\ref{thm:num_approx_finite_time}, about finite time expansion of the numerical scheme~\eqref{def:schema_stoped}. The second part shows the existence of a unique stationary measure for the numerical scheme, as well as Theorem~\ref{thm:num_approx_stationary} for the Talay-Tubaro expansion of this stationary measure.

\subsection{Finite-time expansion for the numerical scheme with rejection}

The reason why we are interested in the scheme~\eqref{def:schema_stoped} instead of the scheme~\eqref{def:schema_Euler} is Lemma~\ref{lem:Lyapunov}. To bound the error made by the numerical scheme, and hence to prove Theorem~\ref{thm:num_approx_finite_time}, we need uniform in $n\in\N$ and $\delta>0$ integrability. To get this integrability, we show that the Lyapunov structure of the continuous process described in~\cite{ergodicity2} still holds after discretization. This comes from the rejection mechanism that prevents the scheme from seeing the part of space where the gradient of the potential and its derivatives are too big compared to $\delta$, which is a neighborhood of the singularities and of infinity. Recall the definition of $\Hd$ from~\eqref{eq:domaine_scheme}, \lj{and denote $\zeta=4d\gamma\beta^{-1}$. To prove Lemma~\ref{lem:Lyapunov}, we first need to introduce a regularized version of the potential $\hat U$ and a modified kinetic energy $\hat W$, along with a numerical scheme corresponding to the counterpart of the scheme~\eqref{def:schema_Euler} for the new Hamiltonian $\hat H(x,y) = \hat U(x) + \hat W(y)$. A comparison between this new numerical scheme with the previously considered one will lead to the result.

\begin{lem}\label{lem:regularized-pot}
Under Assumptions~\ref{assu:basic} and~\ref{assu:growth}, there exist $\delta_0,\ell_0>0$ such that for all $0<\delta<\delta_0$ and $0<\ell<\ell_0$, there exist
\[
\hat U, \hat W:\R^d\to\R_+
\]
satisfying
\begin{itemize}
    \item $\hat U$ and $\hat W$ are bounded.
    \item $\hat U=U$, $\hat W=|\cdot|^2/2$ on $\Hd$.
    \item For all $(x,y)\in\Hd^c$\[\hat \phi(x,y) = \hat U(x) + \hat W(y) + \zeta \frac{\na\hat W\cdot\na \hat U(x)}{1+|\na \hat U(x)|^2}\geqslant 1/\delta^{\ell}.\]
    \item For all $\alpha\in\N^{2d}$, $1\leqslant |\alpha|_1\leqslant 4$, there exists $C_{\alpha}>0$, $\ell_{\alpha}<1/2$ such that: \begin{equation}\label{eq:bound_modified_pot}|\partial^{\alpha} \hat U|, |\partial^{\alpha} \hat W| \leqslant \frac{C_{\alpha}}{\delta^{\ell_{\alpha}}}.\end{equation}
\end{itemize}
\end{lem}

\begin{proof}
For all $0<\ell<1$, $0<\delta<1$, fix a smooth, increasing and concave function \[g_{\delta,\ell}:\R_+\to\R_+\] such that
\[
g_{\delta,\ell}(t) = t,\, \forall t\in [0,4\delta^{-\ell}];\quad g'_{\delta,\ell}(t) \leqslant \frac{2(4\delta^{-\ell})^2}{t^2},\quad |g^{(p)}_{\delta,\ell}(t)| \leqslant \frac{C}{t^{p+1}},\quad \forall t\geqslant 4\delta^{-\ell},\, p\in \llbracket 2,4\rrbracket,
\] for some $C>0$. Then define \[\hat U = g_{\delta,\ell}(U),\quad \hat W = g_{\delta,\ell}(|\cdot|^2/2).\]Let's show that $\hat U$ and $\hat W$ satisfy the desired conditions. First, the conditions on $g_{\delta,\ell}$ yield that
\[
g_{\delta,\ell} \leqslant 4\delta^{-\ell} + \int_{4\delta^{-\ell}}^{\infty} \frac{2(4\delta^{-\ell})^2}{t^2}\dd t = 12\delta^{-\ell},
\]
so that $\hat U$ and $\hat W$ are bounded. For all $(x,y)\in\Hd$, we have using Young's inequality:
\[
\delta^{-\ell} \geqslant U(x) + \frac{|y|^2}{2} + \zeta \frac{y\cdot \na U(x)}{1 + |\na U(x)|^2} \geqslant \frac{1}{2}H(x,y) - \zeta^2,
\]
so that $H \leqslant 2\delta^{-\ell} + 2\zeta^2$. As soon as $\delta^{-\ell}\geqslant\zeta^2$, we have that for all $(x,y)\in\Hd$, $U(x) \leqslant 4\delta^{-\ell}$ and $|y|^2/2 < 4\delta^{-\ell}$, yielding $\hat U(x), \hat W(y) = U(x), |y|^2/2$ and the second point.

The definition of the modified kinetic energy implies that $\na \hat W(y) = g'_{\delta,\ell}(|y|^2/2) y $, and the conditions on $g_{\delta,\ell}$ then yield
\[
\sup_{y\in\R^d}|\na \hat W(y)| \leqslant \max\po \sqrt{8\delta^{-\ell}}, \sup_{s\geqslant \sqrt{8\delta^{-\ell}}} g'(s) s \pf \leqslant 16 \delta^{-\ell/2}.
\]
To prove the third point of the lemma, we distinguish two cases: if $(x,y)\notin \Hd$ satisfies $U(x)<4\delta^{-\ell}$ and $|y|^2/2<4\delta^{-\ell}$, then $\hat\phi(x,y) = \phi(x,y) \geqslant \delta^{-\ell}$ by definition of $\Hd$. If $U(x)>4\delta^{-\ell}$ or $|y|^2/2>4\delta^{-\ell}$, we have
\[
\hat \phi(x,y) \geqslant 4\delta^{-\ell} - 16\zeta \delta^{-\ell/2} \geqslant \delta^{-\ell}, 
\]
for $\delta$ small enough depending on $\ell>0$. The last point is a direct consequence of the conditions imposed on $g_{\delta,\ell}$ and the definitions of $\hat U$ and $\hat W$.
\end{proof}

We are now able to prove:
\begin{lem}\label{lem:Lyapunov}
Under Assumptions~\ref{assu:basic} and~\ref{assu:growth}, there exist $\delta_0,\ell_0>0$ such that for all $0<b<\beta$, there exists $V_b:\mathcal X\to \R_+$ such that:
\begin{itemize}
    \item $\inf_{\X} V_b >0$.
    \item For all $\varepsilon >0$, there exists $C>0$ such that \[C^{-1}e^{(1-\varepsilon)bH} \leqslant V_b \leqslant Ce^{(1+\varepsilon)bH}.\] 
    \item There exists $\alpha,K>0$ such that for all $0<\delta<\delta_0$, $0<\ell<\ell_0$, $(x,y)\in\Hd$:\[\E_{(x,y)}\left( V_b(\bar X_1, \bar Y_1) \right) \leqslant (1-\alpha \delta)  V_b(x,y) + \alpha\delta K.\]
\end{itemize}
\end{lem}

\begin{proof}
Fix $0<b<\beta$ and write:
\[
V_b(x,y) = \exp\left(b\phi(x,y)\right),\quad \hat V_b(x,y) = \exp\left(b\hat \phi(x,y)\right),
\]
where $\phi$ was defined in~\eqref{eq:phi} and $\hat \phi$ in Lemma~\ref{lem:regularized-pot}. It is different from the Lyapunov function of~\cite{ergodicity2} in that we need in the denominator of $\phi$ the term $1 + |\na U(x)|^2$ instead of $|\na U(x)|^2$. However, the proof written in~\cite{ergodicity2} can be immediately adapted to cover the case we consider, and it yields that there exists a compact set $K_1\subset \R^{2d}$ such that for all $(x,y)\notin K_1$:
\begin{equation}\label{eq:Lyapunov_explicit}
\frac{LV_b(x,y)}{V_b(x,y)} \leqslant -c,
\end{equation}
where $L$ is the generator defined in~\eqref{eq:generateur}. For any $\varepsilon>0$, we have 
\[
\left| \zeta \frac{y\cdot\na U(x)}{1+|\na U(x)|^2} \right| \leqslant \varepsilon |y|^2/2 + \frac{\zeta^2}{2\varepsilon}\times \frac{|\na U(x)|^2}{1+|\na U(x)|^2} \leqslant \varepsilon H + C,
\]
for some $C>0$, and this implies the first two points. Now all that remains to do to show the third point is to show that inequality~\eqref{eq:Lyapunov_explicit} remains true for the numerical scheme, as it stays in the region of space where $U=O( \delta^{-\ell})$. Recall the definition of $\hat U$ and $\hat W$ from Lemma~\ref{lem:regularized-pot}. We introduce the following numerical scheme, counterpart of the scheme~\eqref{def:schema_Euler} for the Hamiltonian $\hat H$:
\begin{equation}\label{def:schema_regu}
    \left\{ \begin{aligned} &\hat X_{n+1} = \hat X_n + \delta \na \hat W(Y_{n}), \\ &\hat Y_{n+1} = \hat Y_n -\delta\nabla \hat U(\hat X_n) - \delta\gamma \na \hat W(\hat Y_{n+1})  + \sqrt{2\gamma\beta^{-1}\delta}G_n, \end{aligned}\right.
\end{equation}
where $(G_n)$ is the same family of Gaussian random variable as in the definition of the scheme~\eqref{def:schema_stoped}. 
On the event $\left\{(\bar X_1,\bar Y_1)\neq (\hat X_1,\hat Y_1)\right\}$, the step of $(\bar X_1,\bar Y_1)$ was rejected, $(\hat X_1,\hat Y_1)\notin\Hd$, and we have using the third point of Lemma~\ref{lem:regularized-pot}
\begin{equation}\label{eq:comparaison-schemes}
\hat V_b(\bar X_1,\bar Y_1) = \hat V_b(x,y) \leqslant e^{b\delta^{-\ell}} \leqslant \hat V_b(\hat X_1,\hat Y_1).
\end{equation}
Suppose we have $K>0$, and $0<\alpha<1$ both independent of $\delta$ such that:
\[
\E_{(x,y)}\left( \hat V_b(\hat X_1, \hat Y_1) \right) \leqslant (1-\alpha \delta) \hat V_b(x,y) + \alpha\delta K,
\]
for all $(x,y)\in\Hd$.
Then, using that for all $(x,y)\in\Hd$, $V_b(x,y) = \hat V_b(x,y)$, and inequality~\eqref{eq:comparaison-schemes}, we would have for all $(x,y)\in\Hd$:
\begin{multline*}
  \E_{(x,y)}\left(  V_b( \bar X_1, \bar Y_1) \right) = \E\left( \hat V_b( \bar X_1, \bar Y_1) \right)  \leqslant  \E\left( \hat V_b(\hat X_1, \hat Y_1) \right) \\ \leqslant (1-\alpha \delta) \hat V_b(x,y) + \alpha\delta K = (1-\alpha \delta) V_b(x,y) + \alpha\delta K.
\end{multline*}
Thus we only have to show that the Lyapunov structure holds (partially) true for the regularized numerical scheme~\eqref{def:schema_regu}. In the sequel of this proof, \lj{the notation $o(\delta)$ will denote a quantity uniform in $(x,y)\in\R^{2d}$:
\[
f_{\delta}(x,y) = o(\delta) \Leftrightarrow \lim_{\delta\rightarrow0}\sup_{(x,y)\in\R^{2d}}\frac{\left|f_{\delta}(x,y)\right|}{\delta} = 0.
\]}
We have:
\begin{multline*}
 \hat V_b(\hat X_1,\hat Y_1) -  \hat V_b(x,y) =  \hat V_b(x,y)\left( \hat \phi(\hat X_1,\hat Y_1) - \hat \phi(x,y)   \right)\\ +\frac{1}{2}\hat V_b(x,y)\left( \hat \phi(\hat X_1,\hat Y_1) - \hat \phi(x,y)   \right)^2 + \hat V_b(x,y)\sum_{p\geqslant 3} \frac{1}{p!}\left( \hat \phi(\hat X_1,\hat Y_1) - \hat \phi(x,y)   \right)^p .  
\end{multline*}
A Taylor expansion on $\hat \phi$ yields that there exists $(\theta_1,\theta_2)\in\R^{2d}$, random, such that:
\begin{equation}\label{eq:taylor}
\hat \phi(\hat Z_1) = \hat \phi(z) + ( \hat Z_1-z)\cdot\na\hat \phi(z) + \frac{1}{2}\na^2\hat \phi(z)(\hat Z_1-z)^{\circ 2} + \frac{1}{6}\na^3\hat\phi(\theta_1,\theta_2)(\hat Z_1-z)^{\circ 3},
\end{equation}
where $z=(x,y)$ and $\hat Z_1=(\hat X_1,\hat Y_1)$. Condition~\eqref{eq:bound_modified_pot} yields that for $p\in\left\{1,2\right\}$:
\[
\E\po \po\na^3\hat\phi(\theta_1,\theta_2)(\hat Z_1-z)^{\circ 3}\pf^p \pf= o(\delta).
\]
Moreover, we have
\[
( \hat Z_1-z)\cdot\na\hat \phi(z) = \delta \na \hat W(y)\cdot\na_x\hat \phi -\delta \na \hat U(x)\cdot\na_y\hat\phi - \delta \gamma \na \hat W(y)\cdot\na_y\hat\phi + \sqrt{2\gamma\beta^{-1}\delta}G_1\cdot\na_y\hat\phi(x,y),
\]
and 
\[
\frac{1}{2}\na^2\hat \phi(z)(\hat Z_1-z)^{\circ 2} = \gamma\beta^{-1}\delta G_1\cdot \na^2_y\hat\phi(x,y)G_1 + o(\delta).
\]
Thus, taking the expectation of equation~\eqref{eq:taylor} yields
\begin{multline*}
\E\left( \hat \phi(\hat X_1,\hat Y_1) - \hat \phi(x,y)   \right) \\ = \delta \na \hat W(y)\cdot\na_x\hat \phi -\delta \na \hat U(x)\cdot\na_y\hat\phi - \delta \gamma \na \hat W(y)\cdot\na_y\hat\phi + \delta \gamma\beta^{-1}\Delta_y\hat\phi + o(\delta), 
\end{multline*}
as well as
\[
\E\po\left( \hat \phi(\hat X_1,\hat Y_1) - \hat \phi(x,y)   \right)^2\pf = 2 \delta\gamma\beta^{-1} |\na_y\phi|^2 + o(\delta).
\]
In the same spirit, the bounds~\eqref{eq:bound_modified_pot} and equation~\eqref{eq:taylor} yield:
\begin{align*}
    \E\left( \left( \hat \phi(\hat X_1,\hat Y_1) - \hat \phi(x,y)   \right)^p   \right) & \leqslant \sum_{\underset{k \text{ odd}}{k=0}}^p \binom{p}{k}(C\sqrt{\delta})^k \frac{k!}{2^{k/2}(k/2)!}(C\delta^{1-1/\ell})^{p-k} \\ &\leqslant C^p\sqrt{\delta}^p p!\sum_{\underset{k \text{ odd}}{k=0}}^p \frac{1}{(p-k)!(k/2)!2^{k/2}} \\ &\leqslant C (C\delta)^{p/2} p!.
\end{align*}
Hence for $\delta<C^{-1}$:
\[
\E\left( \sum_{p\geqslant 3} \frac{1}{p!}\left( \hat \phi(\hat X_1,\hat Y_1) - \hat \phi(x,y)   \right)^p   \right) \leqslant C \frac{(C\delta)^{3/2}}{1-C\delta} =  o(\delta).
\]
Finally, using that $\hat U = U$ and $\hat W = |\cdot|^2/2$ on $\Hd$, we get that for all $(x,y)\in \Hd$:
\begin{equation}\label{eq:Lyapunov_discret}
\E_{(x,y)}\po\hat V_b(\hat X_1,\hat Y_1)\pf -  \hat V_b(x,y) =  \hat V_b(x,y)\left(\delta \frac{LV_b(x,y)}{V_b(x,y)} + o(\delta) \right),
\end{equation}
where $o(\delta)$ is uniform in $(x,y)$. Hence, from equation~\eqref{eq:Lyapunov_explicit}, we have $\delta_0>0$ such that for all $(x,y)\in\Hd\setminus K_1$, $\delta<\delta_0$:
\[
\E_{(x,y)}\po\hat V_b(\hat X_1,\hat Y_1)\pf \leqslant (1-\alpha\delta)V_b(x,y),
\]
for some $\alpha>0$. The map $LV_b/V_b$ is continuous on the compact set $K_1$, hence is bounded, and inequality~\eqref{eq:Lyapunov_discret} yields that there exists $C>0$ such that for $(x,y)\in K_1$:
\[
\E_{(x,y)}\po\hat V_b(\hat X_1,\hat Y_1)\pf \leqslant V_b(x,y) + C\delta 
\]
which concludes the proof.
\end{proof}
}

This Lyapunov structure has an immediate corollary: the numerical scheme will stay away from the singularities uniformly on $\delta$, even though its forbidden area is a level set depending on $\delta$.

\begin{cor}\label{cor:integrability_scheme}
Under Assumptions~\ref{assu:basic} and~\ref{assu:growth}, there exist $\delta_0,\ell_0>0$ such that for all $0<b<\beta$, $0<\ell<\ell_0$, $z\in\mathcal X$:
\[
\sup_{0<\delta<\delta_0}\sup_{n\in\N} \E_{z}\left(e^{bH(\bar Z_n)}\right) < \infty. 
\]
Moreover, for all $z\in\mathcal X$, there exists $C,c>0$ such that for all $a>0$, $0<\delta<\delta_0$, $0<\ell<\ell_0$ and $n\in\N$:
\[
\mathbb P_{z}\left( H(\bar Z_n)\geqslant a \right) \leqslant Ce^{-ca}.
\]
\end{cor}

\begin{proof}
\lj{Fix $b<b'<\beta$. The second Lyapunov property of Lemma~\ref{lem:Lyapunov} implies that:
\[
\E_{z}\po e^{bH\po \bar X_n,\bar Y_n\pf}\pf \leqslant C\E_z\po V_{b'}\po \bar X_n, \bar Y_n \pf \pf,
\]
for some $C>0$, whereas an induction argument using the third property yields
\[
\E_z\po V_{b'}\po \bar X_n, \bar Y_n \pf \pf \leqslant \max\po V_{b'}(x,y),K \pf.
\]
Both inequality combined then gives the first point of the Corollary.}
The second point is a direct application of Markov inequality.
\end{proof}

The next lemma will allow us to do series expansion of $P_t(\bar Z_n)$ to prove Theorem~\ref{thm:num_approx_finite_time}. We write for $z_1,z_2\in\R^{2d}$, $[z_1,z_2]=\left\{ tz_1 + (1-t)z_2, t\in[0,1] \right\}$.

\begin{lem}\label{lem:control_rest}
For all $\kappa>0$, there exists $C,\delta_0>0$ such that for all $0<\delta<\delta_0$, $z_1, z_2\in \left\{ H \leqslant \kappa \ln{1/\delta} \right\}$, and $|z_1-z_2|\leqslant \delta^{\frac{2}{3}}$, all $z\in[z_1,z_2]$:
\[
H(z) \leqslant H(z_1) + C 
\]
\end{lem}

\begin{proof}
Fix $z_1,z_2\in \left\{ H \leqslant \kappa \ln{1/\delta} \right\}$, such that $|z_1-z_2|\leqslant \delta^{\frac{2}{3}}$. We have that $[z_1,z_2]\subset B(z_1,\delta^{\frac{2}{3}})$. We first show that \[B(z_1,\delta^{\frac{2}{3}}) \subset \left\{ H \leqslant\kappa' \ln{1/\delta} \right\},\] for some $\kappa'>0$. To this end, fix $z\in B(z_1,\delta^{\frac{2}{3}})$, and write for $0\leqslant s\leqslant 1$:
\[
\varphi(s) = U(z_1 + s(z-z_1)).
\]
Thanks to Assumption~\ref{assu:basic}, we have that $\varphi$ is finite and smooth for $s<s_0$ small enough. We show at the same time that we may have $s_0=1$ and that $\varphi(1)\leqslant \kappa' \ln(1/\delta)$. We have:
\begin{multline*}
\varphi'(s) = \na U(z_1 + s(z-z_1)) \cdot (z-z_1) \\ \leqslant |\na U(z_1 + s(z-z_1))|  |(z-z_1)| \leqslant \delta^{\frac{2}{3}}\sqrt{c_0\varphi^{2+\frac{2}{\eta_0}}+d_0} \leqslant \delta^{\frac{2}{3}} \sqrt{c_0}\po \varphi + \po\frac{d_0}{c_0}\pf^{\frac{\eta_0}{2+2\eta_0}} \pf^{1+\frac{1}{\eta_0}}.
\end{multline*}
Hence there is $c_1,c_2>0$ independent of $z_1,z$ and $\delta$ such that for $s\leqslant s_0$:
\[
\frac{\varphi'(s)}{\po \varphi + c_2 \pf^{1+\frac{1}{\eta_0}}} \leqslant c_1\delta^{\frac{2}{3}}.
\]
Integrating this last inequation yields:
\[
\frac{\eta_0}{\po \varphi(s) + c_2 \pf^{1/\eta_0}} \geqslant \frac{\eta_0}{\po \varphi(0) + c_2 \pf^{1/\eta_0}} - c_1\delta^{\frac{2}{3}}s.
\]
Whether $\eta_0<0$ or $\eta_0>1$, because $\varphi(0)\leqslant \kappa\ln(1/\delta)$, we have that there exists $\delta_0,\kappa'>0$ such that for all $0<\delta<\delta_0$, $\varphi(1)\leqslant \kappa'\ln(1/\delta)$. Hence using Assumption~\ref{assu:growth} again, we have that \[\forall\, z\in \B(z_1,\delta^{\frac{2}{3}}),\, |\na U(z)|\leqslant \kappa''(\ln(1/\delta))^{1+\frac{1}{\eta_0}},\]for some $\kappa''>0$. Using a Taylor expansion on $\varphi$, we get that for all $z_1\in\left\{ H \leqslant \kappa \ln{1/\delta} \right\}$ and $z\in B(z_1,\delta^{\frac{2}{3}})$:
\[
U(z) \leqslant U(z_1) + \kappa''\delta^{\frac{2}{3}}(\ln{1/\delta})^{\frac{1}{\eta_0}},
\]
and this concludes the proof with $C=\sup_{0<\delta<\delta_0}\kappa''\delta^{\frac{2}{3}}(\ln{1/\delta})^{\frac{1}{\eta_0}}$.
\end{proof}

\begin{proof}[Proof of theorem~\ref{thm:num_approx_finite_time}]
As in~\cite{Talay-Tubaro}, we only prove the case $k=1$ for simplicity, since for $k>1$, the proof is the same. Fix some function $f\in\mathcal P(U)$, and write $u(t,x,y) = P_tf(x,y)$. Recall that $u$ is solution to:
\begin{equation}\label{eq:kolmogorov}
    \left\{\begin{aligned}
    &\partial_t u = Lu,\\ &u(t=0,\cdot)=f,
    \end{aligned} \right.
\end{equation}
where $L$ is the generator defined by~\eqref{eq:generateur}. For the sequel of the proof, fix $t>0$. For all $n\in\N$, $\delta>0$ is such that $n\delta = t$. We can write the total error made by the numerical scheme as:
\begin{align*}
\E_z\left(f\left(\bar Z_n\right)\right) - \E_z\left(f\left(Z_t\right)\right) &= \E_z\left( u\left(0,\bar Z_n\right) - u\left(n\delta,z\right)\right) \\ &= \sum_{p=0}^{n-1}\E_z\left( u\left(p\delta,\bar Z_{n-p}\right) - u\left((p+1)\delta,\bar Z_{n-(p+1)}\right)\right). 
\end{align*}
To make a series expansion of the expectations in the sum, we need to distinguish between three cases. Fix some $\kappa >0$ to be determined later, $n,p\in\N$, $p\leqslant n$, and denote:
\begin{align*}
\mathcal A_1 &= \left\{ Z_{n-p},Z_{n-(p+1)}\in\left\{ H\leqslant \kappa \ln(1/\delta)\right\},|Z_{n-p}-Z_{n-(p+1)}|\leqslant \delta^{2/3}\right\}, \\
\mathcal A_2 &= \left\{ Z_{n-p},Z_{n-(p+1)}\in\left\{ H\leqslant \kappa \ln(1/\delta)\right\},|Z_{n-p}-Z_{n-(p+1)}|\geqslant \delta^{2/3}\right\}, \\
\mathcal A_3 &= \left\{ Z_{n-p}\notin\left\{ H\leqslant \kappa \ln(1/\delta)\right\}\right\}\cup \left\{ Z_{n-(p+1)}\notin\left\{ H\leqslant \kappa \ln(1/\delta)\right\}\right\}.
\end{align*}
The event $\mathcal{A}_1$ is the event on which we may do a series expansion of \[u\left(p\delta,\bar Z_{n-p}\right) - u\left((p+1)\delta,\bar Z_{n-(p+1)}\right),\] thanks to Lemma~\ref{lem:control_rest}. Let's first show that the two other event have very low probability as $\delta$ goes to $0$.
Theorem~\ref{thm:estimates} yields the existence of some $0<b<\beta$, $C,q>0$ such that:
\[|u-\int_{\X} f\dd\mu| \leqslant Ce^{-qt}e^{bH}.\]
Fix $q_1,q_2\geqslant 1$ such that 
\[
1/q_1 + 1/q_2 = 1,\qquad bq_1< \beta .
\]
Using Holder inequality we get:
\begin{multline*}
\left|\E_z\left(\left( u\left(p\delta,\bar Z_{n-p}\right) - u\left((p+1)\delta,\bar Z_{n-(p+1)}\right)\right)\mathbbm{1}_{\mathcal A_2}\right)\right| \\ \leqslant Ce^{-qp\delta} \left( \left(\E_z\left( e^{bq_1H\po\bar Z_{n-p}\pf}\right)\right)^{1/q_1} + \left(\E_z\left( e^{bq_1\bar H\po Z_{n-(p+1)}\pf}\right)\right)^{1/q_1} \right) \po\mathbb P\left( \mathcal A_2\right)\pf^{1/q_2}.
\end{multline*}
On the event $\left\{ H\left(\bar Z_{n-(p+1)}\right)\leqslant \kappa\ln(1/\delta) \right\}$, there is some $c>0$ such that 
\[
|\na U\left(\bar X_{n-(p+1)}\right)| + |Y_{n-(p+1)}| \leqslant C\po\ln(1/\delta)\pf^c
\]
and 
\[
\left|\bar Z_{n-k}-\bar Z_{n-(k+1)}\right| \leqslant C\delta \ln(1/\delta)^c + \sqrt{\delta} \left|G\right|
\]
where $G$ is a standard Gaussian random variable. Hence for $\delta$ small enough, there is $c,C>0$ such that:
\[
\mathbb P\left( \mathcal A_2\right) \leqslant \mathbb P\left( G \geqslant c\delta^{1/6} \right) \leqslant Ce^{-c\delta^{1/6}},
\]
and finally
\[
\left|\E_z\left(\left( u\left(p\delta,\bar Z_{n-p}\right) - u\left((p+1)\delta,\bar Z_{n-(p+1)}\right)\right)\mathbbm{1}_{\mathcal A_2}\right)\right| \leqslant Ce^{-c\delta^{1/6}}.
\]
Next, using Holder's inequality again:
\begin{multline*}
\E_z\left(\left( u\left(p\delta,\bar Z_{n-p}\right) - u\left((p+1)\delta,\bar Z_{n-(p+1)}\right)\right)\mathbbm{1}_{\mathcal A_3}\right) \\ \leqslant Ce^{-qp\delta} \left( \left(\E\left( e^{bq_1H\left(\bar Z_{n-p}\right)}\right)\right)^{1/q_1} + \left(\E\left( e^{bq_1\bar H\left(Z_{n-(p+1)}\right)}\right)\right)^{1/q_1} \right)  \left( \mathbb P\left( \mathcal A_3 \right) \right)^{1/q_2}.
\end{multline*}
Fix $b'<\beta$. We have, using Markov inequality:
\[
\mathbb P\left( \mathcal A_3 \right) \leqslant 2\sup_n \mathbb P\po e^{b'H(Z_n)} \geqslant e^{-b'\kappa \ln(1/\delta)} \pf \leqslant 2 \sup_n \E( e^{b'H(Z_n)} ) \delta^{b'\kappa} = C\delta^{b'\kappa}.
\]
We can now fix $\kappa>3q_2/b'$ to get that:
\[
\E_z\left(\left( u\left(p\delta,\bar Z_{n-p}\right) - u\left((p+1)\delta,\bar Z_{n-(p+1)}\right)\right)\mathbbm{1}_{\mathcal A_3}\right) = O(\delta^3).
\]
On the event $\mathcal A_1$, we may use Lemma~\ref{lem:control_rest} to make a series expansion of $u$. For $0\leqslant s \leqslant 1$, write
\[
\varphi(s) = u\left((p+1-s)\delta, \bar Z_{n-(p+1)} + s(\bar Z_{n-p}-\bar Z_{n-(p+1)}\right).
\]
Then:
\[
(\varphi(1)-\varphi(0))\mathbbm{1}_{\mathcal A_1} = \po \sum_{r=1}^5 \frac{\varphi^{(r)}(0)}{r!} + \frac{\varphi^{(6)}(\theta)}{6!}\pf \mathbbm{1}_{\mathcal A_1}, 
\]
for some $0\leqslant \theta \leqslant 1$. Theorem~\ref{thm:estimates} yields that there is $C> 0$, $0<b<\beta$ such that:
\[
|\varphi^{(6)}(\theta)|\mathbbm{1}_{\mathcal A_1} \leqslant C\po\delta^3 G^6 + \delta^{7/2}G^5\ln(1/\delta) + \cdots + \delta^6(\ln(1/\delta))^6\pf e^{bH(\bar Z)}\mathbbm 1_{\mathcal A_1}
\]
for some $\bar Z\in[\bar Z_{n-p},\bar Z_{n-(p+1)}]$. However, Lemma~\ref{lem:control_rest} also yields that \[
H(\bar Z)\mathbbm 1_{\mathcal A_1} \leqslant  H\po\bar Z_{n-p}\pf + C.
\]
Hence, Holder's inequality yields:
\[
\E\po |\varphi^{(6)}(\theta)|\mathbbm{1}_{\mathcal A_1}  \pf \leqslant C\delta^3 \sup_n\E\po e^{q_1bH\po \bar Z_n\pf} \pf^{\frac{1}{q_1}} = C'\delta^3.
\]
Using Holder's inequality as in the bound for the event $\mathcal A_2$ and $\mathcal A_3$, we get that:
\[
\E\po \mathbbm 1_{\mathcal A_1}\sum_{r=1}^5 \frac{\varphi^{(r)}(0)}{r!} \pf = \E\po \sum_{r=1}^5 \frac{\varphi^{(r)}(0)}{r!} \pf + O(\delta^3) = \delta^2\E\po \psi\po(p+1)\delta,\bar Z_{n-(p+1)}\pf\pf + O(\delta^3),
\]
where
\begin{multline*}
    \psi(s,z) = \po \na U + \gamma y \pf \cdot \po \na_y/2 - \na_x\pf  u + \frac{1}{2}\partial_t^2 u - \frac{1}{2}y\cdot \na_x\partial_t u + \frac{1}{2}\po \na U + \gamma y \pf \cdot \na_y^2y \po \na U + \gamma y \pf \\- \frac{1}{2}y\cdot \na_{x,y}^2 \po \na U + \gamma y \pf + \gamma\beta^{-1}\sum_{i=1}^d \partial_{x_i}\partial_{y_i}u - \frac{2}{3}\gamma\beta^{-1} \partial_t\Delta_yu \\- \frac{7}{12}\gamma\beta^{-1} \po \na U + \gamma y \pf\cdot \na_y\Delta_yu + \frac{1}{3}\gamma\beta^{-1}y\cdot\na_x\Delta_y u \\ + \frac{1}{2} \po\gamma\beta^{-1}\pf^2 \sum_{i=1}^d \partial^4_{y_i}u + \frac{1}{6}\po \gamma\beta^{-1} \pf^2 \sum_{i\neq j=1}^d \partial^2_{y_i}\partial_{y_j}^2u.
\end{multline*}
\lj{The expression of $\psi$ is found using the fact that $u$ solves~\eqref{eq:kolmogorov}.} Finally, we have that:
\[
\E_z\left(f\left(\bar Z_n\right)\right) - \E_z\left(f\left(Z_t\right)\right) = \delta^2\E\po \sum_{p=1}^n \psi\po p\delta,\bar Z_{n-p}\pf\pf + O(\delta^2).
\]
The rest follows as in the proof of Talay and Tubaro~\cite{Talay-Tubaro}, which we write for the sake of completeness. As $\psi$ is a function of the derivative of $u$, Theorem~\ref{thm:estimates} yields that:
\begin{equation}\label{eq:weak_conv}
    \left| \E_z\left(f\left(\bar Z_n\right)\right) - \E_z\left(f\left(Z_t\right)\right) \right|  \leqslant C \delta^2 \po \sum_{p=1}^n e^{-qp\delta} \E_z\po e^{bH(\bar Z_{n-p}} \pf+ 1\pf  \leqslant C\delta^2\po \frac{1}{\delta}  + 1 \pf \leqslant C\delta.
\end{equation}
The function $s\rightarrow \E_z\left( \psi(s,Z_{t-s}) \right)$ is smooth and is bounded, hence the weak convergence~\eqref{eq:weak_conv} and classical Riemann summation results yield that:
\[
\left| \delta\E\po \sum_{p=1}^n \psi\po p\delta,\bar Z_{n-p}\pf\pf -\int_0^t \E_z\left( \psi(s,Z_{t-s} \right) \dd s\right| = O(\delta),
\]
which concludes the proof.
\end{proof}

\subsection{Expansion for the invariant measure of the numerical scheme with rejection}

To establish an expansion of the invariant measure of the numerical scheme, we first need to prove its existence, and ergodicity of the scheme. For this purpose, we define a metric on $\mathcal{M}^1\left(\R^{2d}\right)$. Fix some $0<b<\beta$, and write for $z_1,z_2\in\R^{2d}$:
\[
\mathbf d(z_1,z_2) = \mathbbm{1}_{z_1\neq z_2}\po 1+ V_b(z_1) + V_b(z_2)\pf,
\]
where $V_b$ is the Lyapunov function from Lemma~\ref{lem:Lyapunov}. For any two probability measures $\mu, \nu$, we call $(Z_1,Z_2)$ a coupling of $\mu$ and $\nu$ if the law of $Z_1$ (resp. $Z_2$) is $\mu$ (resp. $\nu$). From the distance $\mathbf d$ on $\R^{2d}$, we define the corresponding Kantorovich distance on $\mathcal M^1(\R^{2d})$ by:
\[
W_{\mathbf d}(\mu,\nu) = \inf\left\{\mathbb E\po \mathbf d(Z_1,Z_2)\pf, (Z_1,Z_2) \text{ coupling of }\mu \text{ and }\nu \right\}.
\]
This distance makes $\mathcal M^1(\R^{2d})$ complete. Hence, if we show that the map $\mathcal Law(z)\mapsto \mathcal{L}aw(\bar Z_k)$, where $\bar Z_k$ is the numerical scheme given by~\eqref{def:schema_stoped} with initial condition $z$, for some $k\in \N$, is a contraction, then this would imply the existence of a unique stationary measure, as well as exponentially fast convergence towards it for this Kantorovitch distance. It can be shown that convergence for such a distance would imply ergodicity of the numerical scheme, in the sense of the convergence of the average 
\begin{equation}\label{eq:conv_average}
    \lim_{n\rightarrow\infty}\frac{1}{n}\sum_{k=1}^n\E_{(x,y)}\po f \po\bar X_k,\bar Y_k\pf\pf  =\mu(f).
\end{equation}
Write $K_2=\left\{ H\leqslant K\right\}$, where $K$ is the constant from Lemma~\ref{lem:Lyapunov}. Because $V_b$ is a Lyapunov function, we only need to show that there exists $c>0$, $k\in\N$, such that for all $z\in K_2$, we have the following Doeblin condition:
\begin{equation}\label{eq:Doeblin}
\mathbb P_z\po \bar Z_k \in \cdot \pf \geqslant c\lambda,
\end{equation}
where $\lambda$ stands for the Lebesgue measure, see for example~\cite[Theorem 24]{PDMP}. In other words, we want to show that the numerical scheme creates density.

\begin{lem}\label{lem:ergodicity}
Under Assumptions~\ref{assu:basic} and~\ref{assu:growth}, there exists $\delta_0,\ell_0>0$ such that for all $0<\delta<\delta_0$, $0<\ell<\ell_0$, the numerical scheme~\eqref{def:schema_stoped} admits a unique stationary measure $\mu_{\delta}\in\mathcal M^1(\R^{2d})$, and for all $(x,y)\in\Hd$ and $f:\Hd\to\R$ continuous:
\[
\lim_{n\rightarrow\infty} \frac{1}{n}\sum_{k=1}^n\E_{(x,y)}\po f \po\bar X_k,\bar Y_k\pf  \pf = \mu_{\delta}\po f \pf.
\]
\end{lem}

\begin{proof}
As explained before, it suffices to show the so-called Doeblin condition~\eqref{eq:Doeblin}. To achieve such inequality, we first begin by showing, using a controllability argument, that this is locally true, in the sense that: there exits $c>0$ such that for all $z\in K_2$, $z_0\in B(z,\delta)$ and $A\subset B(z,\delta)$,
\[
\mathbb P_{z_0}\po \bar Z_2 \in A \pf \geqslant c\lambda(A).
\]
To this end, we want to show that there exists $(g_0,g_1)$ such that if $\sqrt{2\gamma\beta^{-1}\delta}(G_0,G_1)=(g_0,g_1)$ and $\bar Z_0=z_0$, then $\bar Z_2 = z_1$. Since Gaussian random variables have a density with respect to the Lebesgue measure, then if $z_1\mapsto (g_0,g_1)$ is a diffeomorphism, the result will hold. Fix $z_0,z_1\in B(z,\delta)$ and write:
\[
F(x,y) = \po x + \delta \po y - \delta\na U(x) - \delta\gamma y  \pf , y - \delta \na U(x) - \delta\gamma y \pf,
\]
and $\mathbf G(g_0,g_1) = \po \mathbf G_1(g_0,g_1),\mathbf G_2(g_0,g_1)\pf$ where
\[
\mathbf G_1(g_0,g_1) = \delta(2-\delta\gamma)g_0 - \delta^2 \na U\po x_0 + \delta \po y_0 -\delta U(x_0) - \delta\gamma y_0 + g_0 \pf \pf + \delta  g_1,
\]
and
\[
\mathbf G_2(g_0,g_1) = (1-\delta\gamma)g_0 - \delta \na U\po x_0 + \delta \po y_0 -\delta U(x_0) - \delta\gamma y_0 + g_0 \pf \pf + g_1.
\]
Now, the goal is to show that there exists $g_0,g_1\in\R^{d}$ such that:
\begin{equation}\label{eq:condi_control}
    H \po F(z_0) + \po \delta g_0, g_0 \pf \pf \leqslant \delta^{-\ell},
\end{equation}
and 
\begin{multline}\label{eq:controle}
\mathbf G(g_0,g_1) = z_1-z_0 \\ +\po \delta^2 (2 - \gamma \delta) \na U(x_0) - \delta (2-3\delta\gamma + (\delta\gamma)^2)y_0 , \delta (1 - \gamma \delta) \na U(x_0) + (2\gamma\delta - (\gamma\delta)^2)y_0  \pf.
\end{multline}
We have:
\[
\mathbf G  = \tilde G \po I_d + \delta \bar G \pf,
\]
where $\tilde G$ is the invertible matrix:
\[
\tilde G = \begin{pmatrix} 2\delta I_d & \delta I_d \\ I_d & I_d \end{pmatrix}, \qquad \tilde G^{-1} = \begin{pmatrix} \delta^{-1} I_d & - I_d \\ -\delta^{-1} I_d & 2I_d \end{pmatrix},
\]
and where $\bar G$ and its derivative are bounded on $B(0,\rho)$, for all $\rho>0$, uniformly with respect to $\delta<\delta_0$ and $z_0\in K_2$. Write $h = I_d + \delta \bar G$. The goal now is to show that for all $\rho_2>0$, there exists $\rho_1>0$ such that $h$ is a diffeomorphism from a neighborhood $W\subset B(0,\rho_1)$ of $0$ to $B(h(0),\rho_2)$, by following the proof of the local inverse theorem. To this end, for $v,g\in\R^{2d}$, write:
\[
\phi_{v}(g) = g - (d_0h)^{-1}(h(g) - v).
\]
There exists $\delta_0>0$ such that, for all $z_0\in K_2$ and $0<\delta<\delta_0$, this function is well defined, i.e. $d_0h$ is invertible. For all $\rho_1>0$, there exists $\delta_0>0$ such that for all $\delta<\delta_0$, $g\in B(0,\rho_1)$, $||| d_g\phi_v||| \leqslant 1/2$. Hence for $g,g'\in B_{\R^{2d}}(0,\rho_1)$, we have 
\[|\phi_{v}(g)-\phi_{v}(g')| \leqslant \frac{1}{2}|g-g'|.\] 
For all $\rho_2>0$, $v\in B(h(0),\rho_2)$, and $\delta$ small enough we have that 
\[|(d_0h)^{-1}(h(0) - v)|<2\rho_2.\]
This implies that if $\rho_1>4\rho_2$, then for all $g\in B(0,\rho_1) $, we have $\phi_{v}(g)\in B(0,\rho_1)$:
\[
|\phi_{v}(g)| \leqslant |\phi_{v}(g)-\phi_{v}(0)| + |\phi_{v}(0)| \leqslant \frac{1}{2}|g| + 2\rho_2 < \rho_1.
\]
Hence we may apply Banach fixed point theorem to get that for all $v\in B(0,\rho_2)$, there exists $g_v\in B(0,\rho_1)$ such that $\phi_{v}(g_v) = g_v$, or equivalently $h(g_v)=v$. The fact that $v\mapsto g_v$ is smooth is proved in the same way as in the proof of the inverse function theorem.
Thus, for all $\rho>0$, there exists $\delta_0,r>0$ such that for all $z_0\in K_2$, $0<\delta<\delta_0$, there exists a bounded neighborhood $W\subset B_{\R^{2d}}(0,r)$ of $0$, such that $G$ is a diffeomorphism from $W$ to $B_{\R^{2d}}(G(0,0),\rho\delta)$. We fix
\[\rho> 3+ 2 \sup_{(x,y)\in K_2}|y| + |\na U(x)|.\]
Write $I(z_0,z_1)$ for the left hand side of Equation~\eqref{eq:controle}. For $\delta$ small enough, $I\in B(G(0,0),\rho\delta)$ for all $z_0,z_1\in B(z,\delta)$, so that Equation~\eqref{eq:controle} has a solution $G^{-1}(I(z_0,z_1))$. The fact that $W\subset B(0,r)$ yields condition~\eqref{eq:condi_control} for $\delta<\Delta/r$, where $\Delta$ is the distance between $K_2$ and $\mathcal D^c$. Finally, because $(G_0,G_1)$ has a positive density, for $z_0\in B(z,\delta)$ and $A\subset B(z,\delta)$ we have:
\begin{multline*}
\mathbb P_{z_0}\po \bar Z_2 \in A \pf \geqslant \mathbb P\po (G_0,G_1) \in G^{-1}\circ I(\cdot,z_0)(A) \pf  \\\geqslant c\int_{G^{-1}\circ I(\cdot,z_0)(A)}\dd z' = c\int_A  |\det(\na G\circ I(z',z_0))|^{-1} \dd z'  \geqslant c'\lambda(A).
\end{multline*}
Now we are able to show condition~\eqref{eq:Doeblin}. Fix $\delta$ small enough so that the local Doeblin condition holds true. Since $K_2$ is compact, there exists $z_1,\cdots,z_N$ such that $K_2= \cup_{i=1}^N B(z_i,\delta)$. For all measurable $A\subset K_2$, we have that 
\[\lambda (A) \leqslant \sum_{i=1}^N \lambda (A\cap B(z_i,\delta)).\] 
Hence we may restrict ourselves to $A\subset B(z_j,\delta)$, for some $j\in \llbracket 1,N \rrbracket$. Up to increasing the value of $K$, we may suppose that $K_2$ is connected. In this case, there exists $k\in\N$ and $\varepsilon>0$ such that for all $z\in B(z_i,\delta)$ and $A\subset B(z_j,\delta)$, there exists a finite sequence $z'_0,z'_1,\dots,z_k'\in K_2$ satisfying the following condition: $z'_0=z$, $z'_k\in B(z_j,\delta)$, and for all $0\leqslant p \leqslant k-1$, there is $i_p$ such that $B(z_p,\varepsilon),B(z_{p+1},\varepsilon) \subset B(z_{i_p},\delta)$. Hence we have: 
\[
\mathbb P \po Z_{2(k+1)} \in A \pf \geqslant \mathbb P \po Z_{2p} \in B(z'_p,\varepsilon), 1\leqslant p\leqslant k, Z_{2(k+1)} \in A \pf \geqslant (cw)^k\lambda(A),  
\]
where $w$ is the volume of $B_{\R^{2d}}(0,\varepsilon)$, and this concludes the proof.
\end{proof}

All function $f\in\mathcal P(U)$ are continuous on $\Hd$, thus we do indeed have the convergence~\eqref{eq:conv_average} for such function.
To prove Theorem~\ref{thm:num_approx_stationary}, we first need weak convergence of the stationary measure of the numerical scheme, for functions that depends on time.

\begin{lem}\label{lem:weak_conv_stationary}
Let $g\in\mathcal C^{\infty}\po \R_+\times\X,\R\pf$ be such that for all $\alpha\in\N^{2d}$, there exists $C_{\alpha},q_{\alpha}>0$ and $b_{\alpha}<\beta$ such that:
\[
|\partial^{\alpha}g(t,z)|\leqslant C_{\alpha}e^{-q_{\alpha}t}e^{b_{\alpha}H(z)}.
\]
Under Assumption~\ref{assu:basic},~\ref{assu:growth} and~\ref{assu:Hr-setting}, there exists $C>0$ such that for $\delta$ small enough:
\[
\left| \mu_{\delta}(g(t,\cdot)) - \mu(g(t,\cdot))  \right| \leqslant Ce^{-qt}\delta.
\]
\end{lem}

\begin{proof}
The proof is very close to the proof Theorem~\ref{thm:num_approx_finite_time}, hence we will omit some details. We fix $g$ satisfying the assumptions of Lemma~\ref{lem:weak_conv_stationary}, $t>0$, and we write $u(s,z)=\E_z(g(t,Z_s))$, where $Z=(X,Y)$ is the continuous process~\eqref{eq:langevin}. $g(t,\cdot)\in\mathcal P(U)$, and we have for any $n\in\N$:
\begin{equation}\label{eq:dev_ergo}
\frac{1}{n}\sum_{k=1}^n\E_z\left( u\left(0,\bar Z_k\right) - u\left(k\delta,z\right)\right) = \frac{1}{n}\sum_{k=1}^n\sum_{p=0}^{k-1}\E_z\left( u\left(p\delta,\bar Z_{k-p}\right) - u\left((p+1)\delta,\bar Z_{k-(p+1)}\right)\right).
\end{equation}
Lemma~\ref{lem:ergodicity} yields that:
\[
\frac{1}{n}\sum_{k=1}^n\E_z\left( u\left(0,\bar Z_k\right)\right) = \frac{1}{n}\sum_{k=1}^n\E_z\left( g \left(t,\bar Z_k\right) \right) \underset{n\rightarrow\infty}{\rightarrow} \mu_{\delta}(g(t,\cdot)).
\]
The continuous process is also ergodic, as shown by Theorem~\ref{thm:Hr-hypo}. Hence we have:
\[
\frac{1}{n}\sum_{k=1}^n\E_z\left( u\left(k\delta,z\right)\right) = \frac{1}{n}\sum_{k=1}^nP_{k\delta}g(t,\cdot)(z) \underset{n\rightarrow\infty}{\rightarrow} \mu(g(t,\cdot)).
\]
As in the proof of Theorem~\ref{thm:num_approx_finite_time}, a Taylor expansion yields:
\[
\left| \sum_{p=0}^{k-1}\E_z\left( u\left(p\delta,\bar Z_{k-p}\right) - u\left((p+1)\delta,\bar Z_{k-(p+1)}\right)\right) \right| \leqslant Ce^{-qt}\delta,
\]
and letting $n$ go to infinity in~\eqref{eq:dev_ergo} concludes the proof.
\end{proof}

\begin{proof}[Proof of Theorem~\ref{thm:num_approx_stationary}]
Fix $f\in\mathcal P(U)$ and write $u(t,z)=\E_z(f(Z_t))$, where $Z=(X,Y)$ is the continuous process~\eqref{eq:langevin}. We write a Taylor expansion of higher order than in the proof of Lemma~\ref{lem:weak_conv_stationary} to get:
\[
\sum_{p=0}^{k-1}\E_z\left( u\left(p\delta,\bar Z_{k-p}\right) - u\left((p+1)\delta,\bar Z_{k-(p+1)}\right)\right) = \delta^2\sum_{p=0}^{k-1}\E_z\left( \psi \left((p+1)\delta, \bar Z_{k-(p+1)} \right) \right) + O(\delta^2),
\]
where the term $O(\delta^2)$ is independent of $k$. Hence we may write:
\begin{align*}
\frac{1}{n}\sum_{k=1}^n\sum_{p=0}^{k-1}\E_z&\left( u\left(p\delta,\bar Z_{k-p}\right) - u\left((p+1)\delta,\bar Z_{k-(p+1)}\right)\right) \\&= \frac{1}{n}\sum_{k=1}^n\sum_{p=0}^{k-1}\delta^2\E_z\left( \psi \left((p+1)\delta, \bar Z_{k-(p+1)} \right) \right) + O(\delta^2) \\&= \frac{\delta^2}{n}\sum_{i=1}^n\sum_{j=0}^{n}\E_z\left( \psi \left((i\delta, \bar Z_{j} \right) \right) - \frac{\delta^2}{n}\sum_{k=1}^n\sum_{p=k}^{n}\E_z\left( \psi \left((p+1)\delta, \bar Z_{k-(p+1)} \right) \right) + O(\delta^2).
\end{align*}
Using Theorem~\ref{thm:estimates} and Corollary~\ref{cor:integrability_scheme}, we bound:
\begin{align*}
    \left| \frac{1}{n}\sum_{k=1}^n\sum_{p=k}^{n}\E_z\left( \psi \left((p+1)\delta, \bar Z_{k-(p+1)} \right) \right) \right| &\leqslant \frac{1}{n}\sum_{k=1}^n\sum_{p=k}^{n} \E_z\left( Ce^{-q(p+1)\delta}e^{bH\po \bar Z_{k-(p+1)} \pf} \right) \\ &\leqslant \frac{C'}{n} \sum_{k=1}^n \sum_{p=k}^{\infty} e^{-q(p+1)\delta} \\ & \leqslant \frac{C''}{n} \sum_{k=1}^n e^{-qk\delta} \underset{n\rightarrow\infty}{\rightarrow} 0.
\end{align*}
Ergodicity of the numerical scheme yields that 
\[
\frac{1}{n}\sum_{j=0}^{n}\E_z\left( \psi \left((i\delta, \bar Z_{j} \right) \right) \underset{n\rightarrow\infty}{\rightarrow} \mu_{\delta}\po \psi(i\delta, \cdot)\pf,
\]
and Theorem~\ref{thm:estimates} and Corollary~\ref{cor:integrability_scheme} that:
\[
\left| \frac{1}{n}\sum_{j=0}^{n}\E_z\left( \psi \left((i\delta, \bar Z_{j} \right) \right) \right| \leqslant Ce^{-qi\delta}.
\]
Hence the dominated convergence theorem yields that:
\[
\frac{1}{n}\sum_{i=1}^n\sum_{j=0}^{n}\E_z\left( \psi \left((i\delta, \bar Z_{j} \right) \right) \underset{n\rightarrow\infty}{\rightarrow} \sum_{i\geqslant 0} \mu_{\delta}\po \psi(i\delta, \cdot)\pf.
\]
Now Lemma~\ref{lem:weak_conv_stationary} and Riemann sum tells us that:
\[
\sum_{i\geqslant 0} \mu_{\delta}\po \psi(i\delta, \cdot)\pf = \int_0^{\infty} \mu(\psi(t,\cdot)) \dd t + O(\delta).
\]
Letting $n$ go to infinity in Equation~\eqref{eq:dev_ergo} (with $g(t,\cdot)=f$) then concludes the proof.
\end{proof}

\subsection*{Acknowledgement.} This works is supported by the French ANR grant SWIDIMS (ANR-20-CE40-0022).






\bibliography{biblio.bib}
\bibliographystyle{plain}

\end{document}